\documentclass[12pt]{article}

\usepackage[top=1.5in, bottom=1.5in, left=1.2in, right=1.2in]{geometry}
\usepackage{graphicx}
\usepackage{epsfig}
\usepackage{float}
\usepackage[active]{srcltx}
\usepackage{amsmath}
\usepackage{amsthm}
\usepackage{amsxtra}
\usepackage{bbm}

\usepackage{amsfonts,amssymb}
\usepackage{latexsym}
\newtheorem{Th}{Theorem}
\newtheorem{Prop}[Th]{Proposition}
\newtheorem{Lm}[Th]{Lemma}
\newtheorem{Co}[Th]{Corollary}

\theoremstyle{definition}
\newtheorem{Def}[Th]{Definition}

\newtheorem{Rem}[Th]{Remark}
\newtheorem{Ex}[Th]{Example}

\newcommand{\supp}{\mathrm{supp}}
\newcommand{\sign}{\mathrm{sign}}

\newcommand{\dd}{\mathrm{d}}
\newcommand{\id}{\mathrm{Id}}

\newcommand{\im}{\mathrm{Im}}
\newcommand{\re}{\mathrm{Re}}

\newcommand{\eg}{\text{e.g.\,}}
\newcommand{\ie}{\text{i.e.\;\,}}

\begin{document}
\title{On spectral properties of the Schreier graphs of the Thompson group $F$}
\author{ {\bf Artem Dudko}  \\
                    IMPAN, Warsaw, Poland  \\
          adudko@impan.pl \\
         {\bf Rostislav Grigorchuk
         } \\        Texas A\&M University, College Station, TX, USA  \\      grigorch@math.tamu.edu }

\date{}

\maketitle

\section{Introduction}
The Thompson's group $F$ is one of the most famous and most important groups related to many areas of mathematics (see the survey \cite{CannonFloydParry-ThompsonGroups-96}). The question about amenability of this group remains open for more than 50 years despite many attempts to solve it. By a remarkable Kesten's criterion of amenability a group $G$ is amenable if and only if 1 belongs to the spectrum of the Markov operator associated to a symmetric random walk on $G$. This criterion also holds for graphs of uniformly bounded degree (see \cite{CeccheriniGrigorchukHarpe-Amenability-99}). Amenability of the Thompson's group $F$ (which is equivalent to amenability of the Cayley graph of $F$) would imply that every Schreier graph of $F$ is amenable. This motivates studying spectral properties of such graphs.

This article is devoted to studying spectral properties of the family of Schreier graphs associated to the action of $F$ on the interval $[0,1]$ (see \cite{CannonFloydParry-ThompsonGroups-96}). A Schreier graph of a group $G$ is determined by a triple $(G,H,S)$, where $H$ is a subgroup of $G$ and $S$ is a system of generators for $G$. These graphs naturally arise in association with an action of $G$ on a set $X$. Given $x\in X$ one can introduce the Schreier graph $\Gamma_x=\Gamma(G,G_x,S)$, where $G_x$ is the stabilizer of $x$.

An interesting and important case is when a subgroup $H<G$ is maximal and of infinite index or when $H$ is weakly maximal (\ie $H$ has infinite index in $G$ and is maximal among the groups with this property). Study of spectral, combinatorial and dynamical properties of Schreier graphs associated with weakly maximal subgroups of branch groups of intermediate growth constructed by the second author got a lot of attention and led to many unexpected results and connections. Among them are relations to holomorphic dynamics (\cite{BartholdiGrigorchuk-Spectrum-00}, \cite{GrigorchukNekrashevych-SchurComplement-07},\cite{DangGrigorchukLyubich-Selfsimilar-20}) and  to random Shr\"odinger operators (\cite{GrigorchukLenzNagnibeda-Spectra-18}). In
\cite{DudkoGrigorchuk-Shape-18} this approach was used to show that there exists an uncountable family of isospectral but not quasi-isometric Cayley graphs.

The new method introduced in \cite{BartholdiGrigorchuk-Spectrum-00} for studying spectra of graphs using the relation to the spectra of operators associated with the Koopman representation for groups acting on rooted trees was generalized by the authors to the case of actions $(G,X,\mu)$ where $\mu$ is a quasi-invariant measure \cite{DudkoGrigorchuk-Spectrum-17}.

Around 2010 the second author suggested to D. Savchuk to inspect Schreier graphs $\Gamma_x$, $x\in[0,1]$, of Thompson group $F$ associated with natural action of $F$ on the interval $[0,1]$ by piecewise-linear transformations. It was discovered by Savchuk that the stabilizers $F_x$ of points $x\in(0,1)$ are maximal subgroups and the Schreier graphs $\Gamma_x$ are amenable. He described their shape. The graph for the case $x=\tfrac{1}{2}$ is shown on Figure \ref{FigSchreierThompson}.
\begin{figure}[h]\centering\includegraphics[width=0.7\linewidth]
{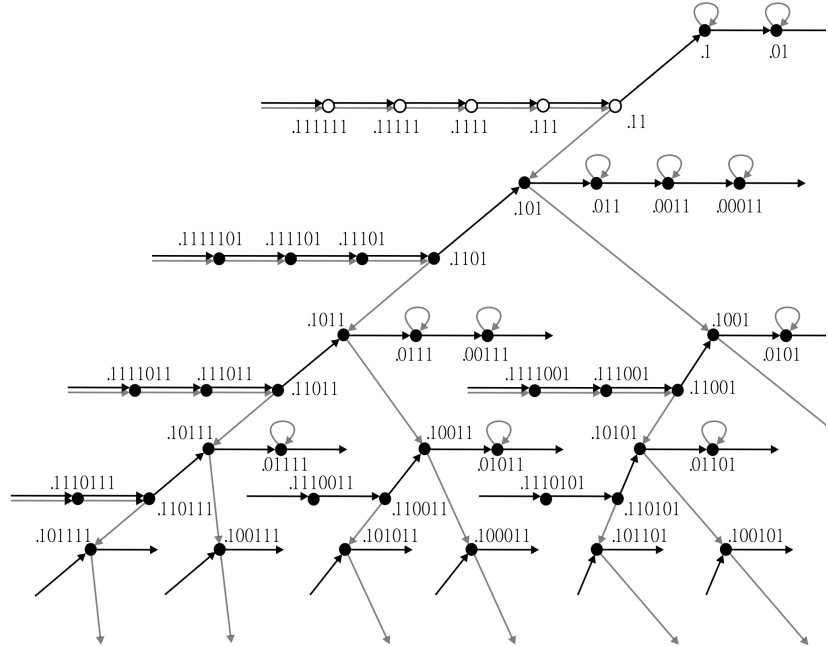}
\caption{The Schreier graph $\Upsilon$ of the action of $F$ on the orbit of $\tfrac{1}{2}$. The black colored edges correspond to the action of the generator $a$ and the grey colored edges correspond to the action of the generator $b$ of $F$.}\label{FigSchreierThompson} \end{figure} This graph (which we denote by $\Upsilon$) was used by Kaimanovich to show that $F$ has no Liouville property \cite{Kaimanovich-ThompsonNotLiouville-17}, \ie there are non-constant harmonic functions on $F$. The graph $\Upsilon$ is one of the main objects of investigation in this article.

 Due to the amenability of Schreier graphs $\Gamma_x$, $x\in[0,1]$, the spectra of Markov operators contain 1, but it is still interesting to know what exactly is the spectrum and what can be said about spectral measures $\mu_v$ associated with delta-functions $\delta_v$, where $v$ is a vertex of the graph. The graphs $\Gamma_x$ associated with $F$ are tree-like graphs. Spectra of tree-like graphs have been studied for many years. Usually, they are unions of finitely many intervals with the spectrum of Lebesgue type (\ie the spectral measure is absolutely continuous with respect to the Lebesgue measure). For instance, in the case of the Cayley graph of a free group $F_m$ the spectrum of the Markov operator associated to the simple random walk is the interval $I_m=[-\tfrac{\sqrt{2m-1}}{m},\tfrac{\sqrt{2m-1}}{m}]$. The spectral measure associated with the delta function $\delta_e$, where $e$ is the identity element of $F_m$, is equal to $d\mu(x)=\frac{\sqrt{2m-1-m^2x^2}}{\pi(1-x^2)}dx$ on $I_m$. In the case of the random walk on the set of integers with probability $1>p>0$ of going one step to the right and probability $q=1-p$ of going one step to the left the spectrum is $[-\sqrt{4pq},\sqrt{4pq}]$ and spectral measure is equal to $d\mu(x)=\frac{dx}{\pi\sqrt{4pq-x^2}}$.

In our study we inspect spectra not only of Markov operators but more generally of so called Laplace type operators  associated to weighted Schreier graphs $\Gamma_x$. Here weight can take negative or even complex values. Given an action of a group $G$ on a set $X$, a point $x\in X$ and an element $m=\sum\limits_{i=1}^n\alpha_i g_i\in\mathbb C[G]$ of the group algebra, where $n\in \mathbb N, \alpha_i\in\mathbb C,g_i\in G$, we consider the associated Laplace type operator on $l^2(Gx)$ denoted by $H_{m,x}$ (see Section \ref{SecPrelim} and formula \eqref{EqHmGen}). Our main results concern the case of the action of the Thompson group $F$ on $[0,1]$ and $m=\alpha(a+a^{-1})+\beta(b+b^{-1})\in\mathbb R[F]$, where $a,b$ are the standard generators of $F$ (see Subsection \ref{SubsecThompson} for details) and $\alpha,\beta\in\mathbb R$. One of our main results is a description of the Kesten spectral measure of $H_{m,1/2}$ with respect to the vertex $1/2$ (Theorem \ref{ThMain}). Another result is a description of the spectrum of $H_{m,x}$ for the case of real $\alpha,\beta$ of the same sign (Theorem \ref{ThSameSign}).

For $\alpha=\beta=1/4$ the operator $H_{m,x}$ coincides with the Markov operator of simple random walk on $\Gamma_x$. Closely related to the Markov operator $M$ is the Laplace operator $\Delta=\id-M$, where $\id$ is the identity operator. Given a vertex $v$ of the Schreier graph $\Gamma_x$ denote by $\eta_v$ the spectral measure of $\Delta$ with respect to the delta function $\delta_v\in l^2(\Gamma_x)$. Let $N_v(s)=\eta_v((-\infty,s])$ be the distribution function of $\eta_v$. These functions play important role in studying Laplace operators. In the case of amenable group the behavior of $N_v(s)$ near zero is of special interest  (see \cite{BendikovPittetSauer}, \cite{Luck-Invariants-01}).
 If $N_v(s)\sim s^\gamma$ near zero then $\gamma$ is called Novikov-Shubin invariant. This parameter is also related to the so-called Lifshitz tails (see \eg \cite{KirschMetzger-IntegratedDensity-07}). For Schreier graphs much less is known. One of our results states that for the Schreier graph $\Upsilon$ (see Figure \ref{FigSchreierThompson}) one has $N_{1/2}(s)\sim s^{3/2}$ near zero (Proposition \ref{PropIntDens}).

For the Cayley graph of the free group $F_m,m\geqslant 2$, or $d$-regular rooted tree $T_d,d\geqslant 2$, probabilities $P_{v_0,v_0}^{(n)}$ of returning to the initial vertex $v_0$ after $n$ steps decay exponentially as follows, for instance, from direct computations or Kesten's criterium of amenability. For $\mathbb Z_+^d$ the decay of these probabilities is of the power type $n^{-\frac{d}{2}}$. When the graph is a combination of parts of non-amenable graphs and amenable graphs, the behavior of return probabilities is unclear in general. We show that for the Schreier graph $\Upsilon$ of the action of Thomspon group $F$ on the orbit of $1/2$ these probabilites decay as $n^{-\frac{3}{2}}$ (Theorem \ref{ThAsympMoments}).

Finally, we address the question about the behavior of spectrum for covering of graphs. In \cite{DudkoGrigorchuk-Shape-18} the authors proved a theorem which they called "Hulanicki type theorem for graphs" (due to its relation to the famous Hulanicki criterion of amenability in terms of weak inclusion of the trivial representation into a regular representation). Namely, they showed that if $\Gamma_1\to\Gamma_2$ is a covering between two weighted graphs of uniformly bounded degree and with uniformly bounded weights then the inclusion of spectra $\sigma(\Delta_2)\subset\sigma(\Delta_1)$ of the corresponding weighted Laplace operators holds in two cases: $1)$ $\Gamma_1$ has subexponential growth; $2)$ $\Gamma_1$ is amenable and $\Gamma_2$ is finite. In Appendix we show that in the second case the inclusion may fail if $\Gamma_2$ is not finite.

The paper is organized as follows. In Section \ref{SecPrelim} we recall the main definitions and formulate the main results. In Section 2 we present some graph operations and describe how they affect the generating series of return probabilities. In particular, we calculate the generating series corresponding to the graph $\Upsilon=\Gamma_{1/2}$ (see Figure \ref{FigSchreierThompson}). In Section \ref{SectionMain} we calculate spectral measures associated to $\Upsilon$ using Cauchy-Stieltjes transform and prove Theorem \ref{ThMain}. In Section \ref{SectionAsymp} we study the generating series associated to simple random walk on $\Upsilon$ and prove Theorem \ref{ThAsympMoments} and Proposition \ref{PropIntDens}. In Section \ref{SectionMuvH} we investigate the dependence of Kesten spectral measure on the vertex of the graph and prove Theorem \ref{ThSameSign}. Finally, in Section \ref{SectionHulanicki} we present an example of graph covering related to Hulanicki Type Theorem for Graphs from \cite{DudkoGrigorchuk-Shape-18} that answers a question raised there.
%%%%%%%%%%%%%%%%%%%%%%%%%%%%%%%%%%%

\subsection*{Acknowledgements}
The authors are grateful to Pierre de la Harpe for valuable comments.

The second author was partially supported
by Simons Foundation Collaboration Grant for Mathematicians, Award Number 527814. Also, the second author acknowledges the support of  the  Max Planck Institute for Mathematics in Bonn and Humboldt Foundation.

\section{Preliminaries and the main results}\label{SecPrelim}
\subsection{Laplace type operators and spectral measures.}
Let $G$ be a finitely generated group with a symmetric (\ie $S=S^{-1}$) generating set $S=\{g_1,\ldots,g_n\}$. Assume that $G$ acts on a set $X$. Given a point $x\in X$ we consider the undirected Schreier graph $\Gamma_x$ whose vertex set is the orbit $Gx$ and the edge set is $\{\{y,g_iy\}:y\in Gx, 1\leqslant i\leqslant n\}$. Given a linear combination of elements from $S$
$$m=\sum\limits_{i=1}^n\alpha_ig_i\in\mathbb C[G]$$ we introduce a Laplace type operator $H_{m,x}$ on $l^2(Gx)$ by
\begin{equation}\label{EqHmGen}(H_{m,x}f)(y)=\sum\limits_{i=1}^n\alpha_if(g_i^{-1}y)\;\;\forall y\in Gx,\end{equation} where $f\in l^2(Gx)$. In particular, if $\alpha_i=\tfrac{1}{n}$ for all $i$ then $H_{m,x}$ is the Markov operator associated to the simple random walk on $\Gamma_x$.

Recall that any self-adjoint operator $A$ on a separable Hilbert space $\mathcal H$ admits a spectral decomposition
$$A=\int\limits_{[-\|A\|,\|A\|]}\lambda E(\lambda),$$ where $E(\lambda)$ is a projection-valued spectral measure defined on Borel subsets of $[-\|A\|,\|A\|]$ and taking values on projections of $\mathcal H$. Given a vector $\xi\in\mathcal H$ one can define a real valued spectral measure $\mu$ on $[-\|A\|,\|A\|]$ (depending on $\xi$) by
\begin{equation}\label{EqSpectralMeasure}\mu(B)=(E(B)\xi,\xi),
\end{equation} where $B$ is a Borel subset of $[-\|A\|,\|A\|]$ and $(\cdot,\cdot)$ is the scalar product on $\mathcal H$.

 In this paper we study the spectral measures $\mu_{m,x}$ of $H_{m,x}$ (in the case when these operators are self-adjoint) associated to the functions $\delta_x\in l^2(\Gamma_x)$. Such measures were introduced by Kesten in \cite{Kesten-BanachMean-59,Kesten-RandomWalk-59} for the case of random walks on groups. Sometimes they are called Kesten spectral measures. In general, computing such spectral measures is a complicated task. One of results in this direction concerns the group of intermediate growth constructed by the second author in \cite{Grigorchuk-Grigorchuk-80}. The spectral measures for the action of this group on the boundary of the binary rooted tree were computed in \cite{GrigorchukKrylyuk-Grigorchuk-12}. This result was generalized to the case of spinal groups \cite{GrigorchukPerezNagnibeda-Schreier-18}.

\subsection{Thompson group $F$.}\label{SubsecThompson}
\noindent It is known that $F$ is generated by two elements $a,b$ where
\begin{equation}\label{EqAB} a(x)=\left\{\begin{array}{ll}\tfrac{x}{2}&\text{if}\;\;x\in[0,\tfrac{1}{2}),\\
x-\tfrac{1}{4}&\text{if}\;\;x\in[\tfrac{1}{2},\tfrac{3}{4}),\\
2x-1&\text{if}\;\;x\in[\tfrac{3}{4},1];\end{array}\right.\;\; b(x)=\left\{\begin{array}{ll}x&\text{if}\;\;x\in[0,\tfrac{1}{2}),\\
\tfrac{1}{2}a(2x-1)+\tfrac{1}{2}&\text{if}\;\;x\in[\tfrac{1}{2},1].\end{array}\right.
\end{equation}
For a point $x\in(0,1)$ one can introduce a Schreier graph $\Gamma_x$ associated to the action of $F$ on the orbit $Fx$. These graphs were described by Savchuk in \cite{Savchuk-ShcreierThompson-10,Savchuk-SchreierThompson-15}. We  compute spectral measures of certain Laplace type operators associated to the graph $\Upsilon=\Gamma_{1/2}$. As a corollary we obtain the description of spectra of these Laplace type operators for $\Gamma_x$ for all $x\in (0,1)$.

\begin{figure}[h]\centering\includegraphics[width=0.7\linewidth]{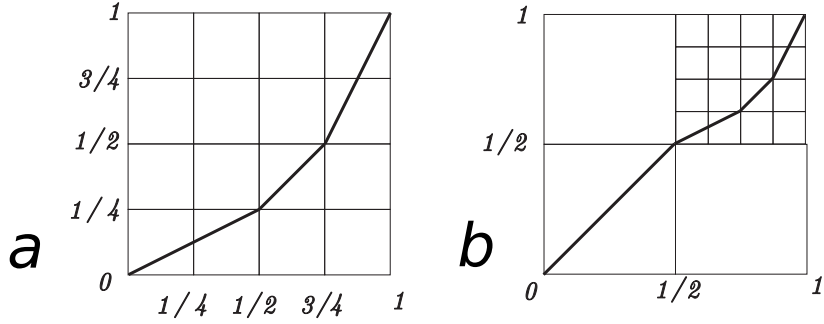}
\caption{The graphs of the generators $a,b$ of $F$}\label{FigThompsonGen} \end{figure}

\subsection{Main results.}
Our main results concern spectra and spectral measures of Laplace type operators $H_{m,x}$ on $l^2(\Gamma_x)$, where \begin{equation}\label{EqmDef}m=\alpha(a+a^{-1})+\beta(b+b^{-1})\in\mathbb R[F].\end{equation}
Fix $\alpha,\beta\in\mathbb R\setminus \{0\}$. Introduce functions $P_*,Q_*$ on $\mathbb R$ by the formulas
\begin{equation}\label{EqPQFormulas0} P_*(z)=\tfrac{z}{2}\pm\tfrac{1}{2}\sqrt{z^2-4(\alpha+\beta)^2},\;\;
Q_*(z)=\tfrac{z}{2}+\beta\pm\tfrac{1}{2}\sqrt{z^2-4\beta z+4(\beta^2-\alpha^2)}.
\end{equation} Here if the expression under the root sign is nonnegative we choose the sign equal to the sign of $z$ in the formula for $P_*(z)$ and choose the sign equal to the sign of $z-2\beta$ in the formula for $Q_*(z)$. If the expression under the root sign is negative we choose the root with a positive imaginary part.

Consider the equation
\begin{equation}\label{EqV0}\beta^4v^3+(P_*(z)+Q_*(z))\beta^2v^2+
(P_*(z)Q_*(z)+(\beta^2-\alpha^2))v+P_*(z)=0.\end{equation}
We will show that for any $z\in\mathbb R$ it has at most one solution $V$ with positive imaginary part. Recall that the support of a measure $\mu$ on $\mathbb R$ (denoted by $\supp(\mu)$) is the set of points $s\in\mathbb R$ such that for every open neighborhood $U$ of $s$ one has $\mu(U)>0$. Below we list the main results of the paper.
\begin{Th}\label{ThMain} The spectral measure $\mu_{m,1/2}$ of the operator $H_{m,1/2}$ on $l^2(\Gamma_{1/2})$ is absolutely continuous with respect to the Lebesgue measure $\lambda$ on $\mathbb R$. The support of $\mu_{m,1/2}$ is the closure of the set of points $z$ for which \eqref{EqV0} has a solution $V=V(z)$ with positive imaginary part. In particular, $\supp(\mu_{m,1/2})$ is a union of at most finitely many closed intervals. For $z\in\supp(\mu_{m,1/2})$ one has:
\begin{equation}\label{EqDmuDlam}\frac{\dd\mu_{m,1/2}(z)}{\dd\lambda(z)}=\im \frac{V(z)}{\pi(\beta V(z)+1)^2}.\end{equation} The Radon-Nikodym derivative above is continuous everywhere and is analytic everywhere except finitely many points $z$.
\end{Th}
\noindent In this paper by $\sigma(A)$ we denote the spectrum of an operator $A$.
\begin{Th}\label{ThSameSign} If $\alpha,\beta$ are of the same sign then for every $x\in (0,1)$ the spectrum of $H_{m,x}$ is an interval and one has \begin{equation}\label{EqSameSignTh}\sigma(H_{m,x})=\supp(\mu_{m,x})=[-2|\alpha+\beta|,2|\alpha+\beta|].\end{equation} \end{Th}
\noindent In addition, we compute the asymptotics of the return probabilities for the simple random walk on $\Upsilon=\Gamma_{\frac{1}{2}}$. As usual, let $O(s)$ stand for a function defined near zero (or infinity) such that $\limsup\ |O(s)/s|$ is finite when $s\to 0$ (respectively, when $s\to\infty$).
\begin{Th}\label{ThAsympMoments} Consider the simple random walk on $\Upsilon$. Let $p_n$ be the probability starting at $\frac{1}{2}$ to return after $n$ steps to $\frac{1}{2}$. Then there exist constants $A,B$ such that one has:
$$p_{2n}=An^{-\frac{3}{2}}+O(n^{-2}),\;\;p_{2n+1}=
Bn^{-\frac{3}{2}}+O(n^{-2}),\;\;\text{when}\;\;n\to\infty.$$
\end{Th}
\noindent
The formula \eqref{EqSameSignTh} shows that in the case of simple random walk (\ie when $m=\tfrac{1}{4}(a+b+a^{-1}+b^{-1})$) one has $\sigma(H_{m,1/2})=\supp(\mu_{m,1/2})=[-1,1]$. The asymptotic behavior of the spectral distribution $N_{1/2}(s)$ of the Laplace operator $\id-H_{m,1/2}$ with respect to the vector $\delta_{1/2}$ is given in the text statement.
\begin{Prop}\label{PropIntDens} For some constant $C>0$ one has: $N_{1/2}(s)=Cs^{\frac{3}{2}}+O(s^2)$ near zero, where $s>0$.
\end{Prop}

As Savchuk showed in \cite{Savchuk-ShcreierThompson-10}, Schreir graphs $\Gamma_x$ of the action of $F$ on $[0,1]$ are amenable. Observe also that the action of $F$ on $[0,1]$ is hyperfinite, \ie the partition on orbits is an increasing union of finite Borel equivalence relations. Indeed, as it was noted in \cite{GrigorchukNekrashevychSushchanskiy-00}, the orbit equivalence relation of the action of $F$ on $[0,1]$ is isomorphic (modulo countable set of points) to the tail equivalence relation on $\{0,1\}^{\mathbb{N}}$ with the standard Borel structure. The latter is hyperfinite by \cite{DoughertyJacksonKechris-94}.

 Given a representation $\pi$ of a group $G$ in a Hilbert space and an element $f\in\mathbb C[G]$ set
$$\pi(f)=\sum\limits_{g\in\supp(f)}f(g)\pi(g).$$ As a corollary of Theorem \ref{ThSameSign} using Theorem 1 from \cite{DudkoGrigorchuk-Spectrum-17} we obtain:
\begin{Co} Let $\kappa$ be the Koopman representation of $F$ on $L^2([0,1],\lambda)$. Let $m$ be as above with $\alpha,\beta\in \mathbb R\setminus \{0\}$ of the same sign. Then $\sigma(\kappa(m))=[-2|\alpha+\beta|,2|\alpha+\beta|]$.
\end{Co}\noindent
As an interesting example we show:
\begin{Prop}\label{Co-11} For $\alpha=-1,\beta=1$ one has:
$$\supp(\mu_{m,1/2})=[z_1,z_2]\cup [0,4],$$ where $z_1\approx -2.766,z_2\approx -0.062$ are the real roots of the polynomial $z^4-2z^2+16z+1$.
\end{Prop}

\section{Weighted graphs and generating functions.}
The aim of this section is to establish the framework for obtaining information about the moments of the spectral measures of the Laplace type operators $H_{m,x}$ described above.
\subsection{General setting.}\label{SubsecGeneral}
Let a countable group $G$ act on a set $X$. Fix an element $m\in\mathbb C[G]$. We will view $m$ as a function from $G$ to $\mathbb C$ and denote by $\supp(m)$ the set of elements $g\in G$ such that $m(g)\neq 0$. Consider a symmetric  ($S=S^{-1})$ generating set $S=\{g_1,g_2,\ldots g_n\}$ that contains $\supp(m)$. Given a point $x\in X$ we consider the corresponding Schreier graph $\Gamma_x$. To study the spectrum of the operator $H_{m,x}$ on $l^2(\Gamma_x)$ it is convenient to use weighted directed graphs.
\begin{Def}\label{DefMarked} For any undirected graph $\Gamma$ we call a directed edge an edge together with a choice of its initial and terminal vertices. By weighted (directed) graph we mean a graph with an assignment of a weight  $\omega_e\in\mathbb C$ for every directed edge $e$.
\end{Def}\noindent Note that every edge of a graph $\Gamma$ gives rise to two directed edges.
Thus, every edge of a graph $\Gamma$ has two weights assigned to it which might or might not coincide. This also concerns the loops. Per our definition, weighted graphs are always directed.
\begin{Def} Given $m\in\mathbb C[G]$ and any Schreier graph $\Gamma$ of $G$ we define by $\Gamma_m$ the weighted graph with weight $m(g)$ assigned to each directed edge $(v,gv)$, $v\in \Gamma,g\in S$.\end{Def}

Given any weighted graph $\Gamma$ of bounded degree with uniformly bounded weight (\ie there exists $D>0$ such that $|\omega_e|\leqslant D$ for each directed edge $e$ of $\Gamma$) one can introduce an associated Laplace type operator $H_\Gamma$ as follows. For a vertex $v$ of $\Gamma$ let $E_v$ denote the set of directed edges ending at $v$. For a directed edge $e$ let $i_e$ stand for its initial vertex. Set
\begin{equation}\label{EqHeckeGamma}(H_\Gamma f)(v)=\sum\limits_{e\in E_v}\omega_e f(i_e)\;\;\forall v\in\Gamma.
\end{equation} Observe that uniform boundedness of the degree and of the weight implies that the operator $H_\Gamma$ is bounded. Moreover, in the case of real symmetric weight the operator $H_\Gamma$ is self-adjoint.
In the case of the weighted graph associated to $m\in \mathbb C[G]$ the above Laplace type operator coincides with the one defined in \eqref{EqHmGen}.

Further, each path $\gamma$ in $\Gamma$ with a chosen starting and terminal vertices can be written as a chain of directed edges $\gamma=(e_1,\ldots,e_k)$. Set $$\omega_\gamma=\prod\limits_{e\in\gamma} \omega_e.$$
For a vertex $v$ of $\Gamma$ denote by $\mathcal P_{\Gamma,v}$ the set of all paths (including trivial) starting and ending at $v$. Denote by $\mathcal F_{\Gamma,v}$ the set of non-trivial paths from $\mathcal P_{\Gamma,v}$ which do not visit $v$ except at starting and ending vertices (the paths of first-return).
Let $|\gamma|$ stand for the length (total number of edges) of the path $\gamma$. For $n\geqslant 0$ set
\begin{equation}\label{EqPGammavnDef}\mathcal P_{\Gamma,v}^{(n)}=\{\gamma\in \mathcal P_{\Gamma,v}:|\gamma|=n\},\;\; p_{\Gamma,v}^{(n)}=\sum\limits_{\gamma\in \mathcal P _{\Gamma,v}^{(n)}}\omega_\gamma.\end{equation} Similarly define
$\mathcal F_{\Gamma,v}^{(n)}$ and $f_{\Gamma,v}^{(n)}$:
\begin{equation}\label{EqFGammavnDef}\mathcal F_{\Gamma,v}^{(n)}=\{\gamma\in \mathcal F_{\Gamma,v}:|\gamma|=n\},\;\; f_{\Gamma,v}^{(n)}=\sum\limits_{\gamma\in \mathcal F _{\Gamma,v}^{(n)}}\omega_\gamma.\end{equation}
Note that in the case when $\omega_e$ are transition probabilities for a random walk on $\Gamma$ the value $p_{\Gamma,v}^{(n)}$ (correspondingly $f_{\Gamma,v}^{(n)}$)  is the probability that starting at $v$ the random walk after $n$ steps returns to $v$ (correspondingly, first time returns to $v$).
For any weighted graph $\Gamma$ one has
\begin{equation}\label{EqpgvHG}p_{\Gamma,v}^{(n)}=(H_\Gamma^n\delta_v,\delta_v)\end{equation} where $\delta_v\in l^2(\Gamma)$ is the Kroneker delta-function at the vertex $v$. In particular, if $H_\Gamma$ is self-adjoint then $p_{\Gamma,v}^{(n)}$ are the moments of the spectral measure of $H_\Gamma$ corresponding to $\delta_v$.

Now, introduce the generating functions
\begin{equation}\label{EqPFDef}P_{\Gamma,v}=P_{\Gamma,v}(t)=\sum\limits_{\gamma\in \mathcal P_{\Gamma,v}}\omega_\gamma t^{|\gamma|},\;\;F_{\Gamma,v}=F_{\Gamma,v}(t)=\sum\limits_{\gamma\in\mathcal F_{\Gamma,v}}\omega_\gamma t^{|\gamma|}.\end{equation} In the case when $\Gamma=\Gamma_m$ is a Schreier graph and $m$ is a symmetric probability distribution on the set of generators the generating function $P_{\Gamma,v}$ is the Green's function of the corresponding random walk on $\Gamma$.
The following relations are well known:
\begin{Lm}\label{LmPF} One has $$P_{\Gamma,v}=\frac{1}{1-F_{\Gamma,v}},\;\;F_{\Gamma,v}=1-\frac{1}{P_{\Gamma,v}}.$$
\end{Lm}
\begin{Rem}\label{RemPFConv} In the case of a graph of a bounded degree with uniformly bounded weight for any $t$ such that $|t|<\|H_\Gamma\|^{-1}$ the series $P_{\Gamma,v}(t)$ and $F_{\Gamma,v}(t)$ converge and moreover
$$P_{\Gamma,v}(t)=((I-tH_\Gamma)^{-1}\delta_v,\delta_v).$$
\end{Rem}
 Further, for a graph $\Gamma$ (possibly, weighted) and a vertex $v$ of $\Gamma$ denote by $\Gamma_v$ the corresponding rooted graph. For a rooted graph $\Gamma_v$ denote by $\delta+\Gamma_v$ the rooted graph obtained by adding a new vertex $\delta$ (serving as a new root) to $\Gamma_v$ connected by a single edge to $v$. In case of a weighted graph we keep the weights on the edges of $\Gamma$ and equip
 the new edge joining $\delta$ and $w$ with new weights $\omega_{(\delta,w)},\omega_{(w,\delta)}\in\mathbb C$. Thus, for a weighted graph $\Gamma$ the result of the operation $"\delta+"$ depends on two complex numbers $\omega_{(\delta,w)}$ and $\omega_{(w,\delta)}$. We do not incorporate these numbers in the notation of the operation to simplify the notations.

 For two graphs $\Gamma$ and $\Delta$ (possibly, weighted) and two vertices $v\in \Gamma,w\in \Delta$ denote by $\Gamma_v\cup\Delta_w$ the rooted graph obtained by taking disjoint copies of $\Gamma$ and $\Delta$ and identifying $v$ and $w$. If $\Gamma$ and $\Delta$ are weighted graphs the graph $\Gamma_v\cup\Delta_w$ is a weighted graph with the weights inherited from $\Gamma$ and $\Delta$. Note that the  operation $"\cup"$ is commutative and associative.

 To describe the Schreier graphs of Thompson group it is convenient to introduce an additional operation $\star$:
\begin{equation}\label{EqStar}
\Gamma_v\star\Delta_w=\Gamma_v\cup(\delta+\Delta_w).
\end{equation} It is not hard to see on simple examples that the operation $"\star"$ is neither commutative nor associative.
Note that, by construction, the results of operations $"\delta +",\;"\cup"$ and $"\star"$ are rooted graphs (weighted if the graphs involved in the operations are weighted). These operations are illustrated on Figure \ref{FigGraphOperations}.
\begin{figure}[h]\centering\includegraphics[width=0.8\linewidth]{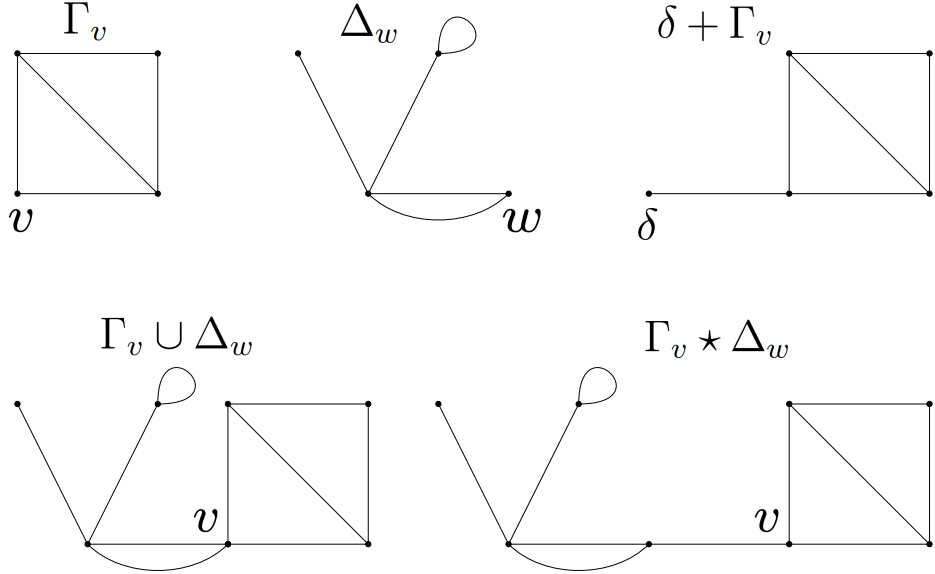}
\caption{Illustration to operations on graphs.}\label{FigGraphOperations} \end{figure}

\begin{Lm}\label{LmGraphSum}
One has
$$F_{\Gamma_v\cup\Delta_u,v}=F_{\Gamma,v}+F_{\Delta,u},\;\;F_{\delta+\Gamma_v,\delta}=
\omega_{(\delta,v)}\omega_{(v,\delta)}t^2P_{\Gamma,v},
\;\;F_{\Gamma_v\star\Delta_u,v}=F_{\Gamma,v}+\omega_{(v,u)}\omega_{(u,v)}t^2P_{\Delta,u}$$
\end{Lm}
\begin{proof} Since $\Gamma$ and $\Delta$ as subgraphs of $\Gamma_v\cup\Delta_u$ intersect only at one vertex $v$ any non-trivial first return path $\gamma\in\Gamma_v\cup\Delta_u$ at the vertex $v$ either belongs to $\Gamma$ or belongs to $\Delta$, and the two possibilities are mutually exclusive. Thus, $\mathcal F_{\Gamma_v\cup\Delta_u,v}$ can be written as a disjoint union of $\mathcal F_{\Gamma,v}$ and $\mathcal F_{\Delta,u}$. By definition of the first-return series \eqref{EqPFDef} this implies the first formula of Lemma \ref{LmGraphSum}.

To prove the second formula of Lemma \ref{LmGraphSum} observe that any first return path $\gamma$ at the vertex $\delta$ in $\delta+\Gamma_v$ starts from the edge $(\delta,v)$, follows a path $\gamma_1$ in $\Gamma$ (starting and ending at $v$), and finishes by the edge $(v,\delta)$.

The third formula of Lemma \ref{LmGraphSum} follows from \eqref{EqStar} and the combination of the first two.
\end{proof}
\subsection{Weighted Schreier graphs of the action of $F$ on $[0,1]$.}\label{SubsecWeightedSchreier}
In this subsection we give a description of the building blocks of Schreier graphs $\Gamma_x$ of the action of $F$ on $[0,1]$ and obtain a description of $\Upsilon=\Gamma_{1/2}$ (see Figure \ref{FigSchreierThompson}) in terms of the operation $\star$ and simpler graphs which are subgraphs of $\Upsilon$. For all graphs we construct below we assume the property that $\omega_{(v,w)}=\omega_{(w,v)}$ for any edge $e=(v,w)$.
\begin{figure}[b]\centering\includegraphics[width=0.85\linewidth]{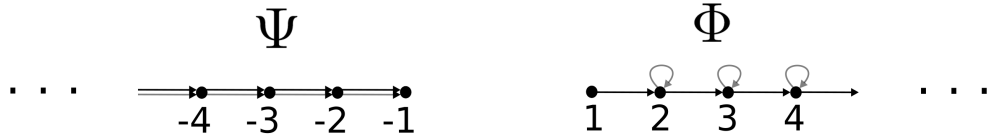}
\caption{Graphs $\Psi$ and $\Phi$. Every black edge has weight $\alpha$ and every grey edge has weight $\beta$.}\label{FigPsiPhi} \end{figure}

Introduce directed graphs $\Phi$ and $\Psi$ with vertex sets labeled by positive and negative integers correspondingly such that:\\
$\bullet$ in $\Phi$ every two consecutive integers are connected by one edge, every integer $i>1$ has a loop attached to it and there are no other edges;\\
$\bullet$ in $\Psi$ every two consecutive integers are connected by two edges and there are no other edges.
\begin{figure}[t]\centering\includegraphics[width=0.7\linewidth]{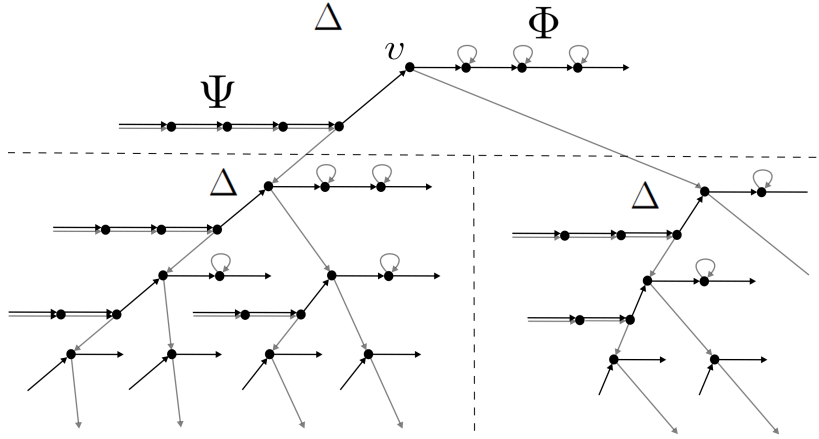}
\caption{Illustration to formula \eqref{EqDv}. Every black edge has weight $\alpha$ and every grey edge has weight $\beta$.}\label{FigSchreierThompsonSelfsim} \end{figure}

Fix two nonzero real numbers $\alpha,\beta$. We make $\Phi$ and $\Psi$ weighted graphs by:\\
$\bullet$ marking all directed "horizontal" (\ie joining consecutive integers) edges in $\Phi$ by $\alpha$ and all directed loops by $\beta$;\\
$\bullet$ for every $i$ marking one edge of $\Psi$ between $i$ and $i+1$ by $\alpha$ (in both directions) and another by $\beta$ (in both directions).

The shapes of graphs $\Phi$ and $\Psi$ are shown at Figure \ref{FigPsiPhi}.
\begin{Lm}\label{LmDeltaEq} There exists a unique connected weighted rooted graph $\Delta_v$ with countably many vertices such that
\begin{equation}\label{EqDv}\Delta_v=(\Phi_1\star\Delta_v)\star(\Psi_{-1}\star\Delta_v)\end{equation} and, in addition, $\omega_{(1,v)}=\omega_{(-1,v)}=\beta$, $\omega_{(-1,1)}=\alpha$.
\end{Lm}
\noindent The graph $\Delta_v$ from \eqref{EqDv} is illustrated at Figure \ref{FigSchreierThompsonSelfsim}.
\begin{proof} Using the equation \eqref{EqDv} one can construct $\Delta_v$ as a union of increasing sequence of weighted rooted graphs $\Gamma_n$ defined inductively. We start with a graph $\Gamma_0$ consisting of a single vertex $v$ which will be the root of $\Delta_v$. On step $n$ we will add some vertices and weighted edges to already constructed graph $\Gamma_n$. Namely, for each leaf $w$ (vertex adjacent to one edge) of $\Gamma_n$ we do the following:
\begin{itemize}
\item{} attach a new copy of $\Phi$ to $w$ so that the vertex $1$ coincide with $w$;
\item{} connect $w$ by an edge with weight $\beta$ to a new vertex $w_1$;
\item{} connect $w$ by an edge with weight $\alpha$ to a vertex $-1$ of a new copy of $\Psi$;
\item{} connect $-1$ by an edge with weight $\beta$ to a new vertex $w_2$.
\end{itemize} Denote the obtained graph by $\Gamma_{n+1}$. Note that for each leaf $w$ of $\Gamma_n$ the vertices $w_1,w_2$ are leaves of $\Gamma_{n+1}$.

It is straightforward to verify that $\Gamma_\infty=\bigcup\limits_{n=1}^\infty\Gamma_n$ satisfies \eqref{EqDv} and any connected weighted rooted graph $\Delta_v$ with countably many vertices satisfying \eqref{EqDv} is isomorphic to $\Gamma_\infty$.
\end{proof}

As before let $a,b$ be the generators of the Thompson group $F$ and $\alpha,\beta\in\mathbb R\setminus\{0\}$. Let $m=m_{\alpha,\beta}=\alpha(a+a^{-1})+\beta(b+b^{-1})\in\mathbb R[F]$. Recall that one of the main objects of investigation of this article is the graph $\Upsilon=\Gamma_{1/2,m},$ which is the weighted Schreier graph of the action of $F$ on $[0,1]$ at point $\frac{1}{2}$ with the weight defined by $m_{\alpha,\beta}$ (see Figure \ref{FigSchreierThompson}).  Using the description of $\Upsilon$ given in \cite{Savchuk-ShcreierThompson-10} (Proposition 1), and \cite{Savchuk-SchreierThompson-15} (Proposition 2.1),  we obtain:
\begin{Co}\label{PropGa1/2} Let $\Delta_v$ be the graph from Lemma \ref{LmDeltaEq}. Then \begin{equation}\Upsilon=\tilde\Phi_1\star(\Psi_{-1}\star\Delta_v)\end{equation} with $\omega_{(-1,v)}=\beta$ and $\omega_{(-1,1)}=\alpha$, where $\tilde\Phi_1$ is obtained from $\Phi_1$ by adding a loop at the vertex $1$ with the weight $\beta$.
\end{Co}
%%%%%%%%%%%%%%%%%%%%%%%%%%%%%%%%%%%%%%%%%%%%%%%%%
\subsection{Computing the generating functions.}\label{SubsecCompGen}
To simplify notations set \begin{equation*}\label{EqqpDef} q(t)=F_{\Phi,1}-1,\;\;p(t)=F_{\Psi,-1}(t)-1,\;\;x(t)=P_{\Delta,v}(t)\;\;(\text{see}\;\;\eqref{EqPFDef}).\end{equation*}
We will drop the variable $t$ where it does not create confusion. For instance, we will write $p,q,x$ instead of $p(t),q(t),x(t)$.
Note that the functions $p(t),q(t)$ are different from the functions $P_*(z), Q_*(z)$ defined in \eqref{EqPQFormulas0} (see also the formula \eqref{EqPQ*}). In particular, the variables $z$ and $t$ are related by $z=t^{-1}$.

By Remark \ref{RemPFConv}, $q(t),p(t)$ and $x(t)$ are convergent for $t$ near the origin, and thus they are elements of the ring $\mathbb C\{t\}$ of germs of complex analytic functions near $0\in\mathbb C$.
\begin{Lm}\label{LmqpForm} In the space $\mathbb C\{t\}$ of germs at the origin one has:
\begin{equation}\label{EqqpForm} q=-\tfrac{1}{2}-\beta t-\tfrac{1}{2}\sqrt{(2\beta t-1)^2-4\alpha^2 t^2},\;\;p=-\tfrac{1}{2}-\tfrac{1}{2}\sqrt{1-4(\alpha+\beta)^2t^2},
\end{equation} where the branch of the radical is chosen so that $\sqrt{1}=1$.
\end{Lm}
\begin{proof} Observe that $\Phi_1$ is isomorphic as rooted graph to $\delta+\tilde\Phi_1$, where $\delta$ is the graph consisting of one vertex without edges. Recall that $\tilde\Phi_1$ is obtained from $\Phi_1$ by adding a loop at $1$ with weight $\beta$. The weight of the edge connecting $\tilde\Phi$ to $\delta$ is equal to $\alpha$. In the graph $\tilde\Phi$ any $1$-first-return path is either a $1$-first-return path of $\Phi$ or is equal to the loop at $1$ in one of the two possible directions. It follows that \begin{equation}\label{EqTilPhi}F_{\tilde\Phi,1}=2\beta t+F_{\Phi,1}.\end{equation} Using Lemmas \ref{LmGraphSum} and \ref{LmPF} we obtain:
$$F_{\Phi,1}=\alpha^2t^2P_{\tilde\Phi,1}=\frac{\alpha^2t^2}{1-F_{\tilde\Phi,1}}=\frac{\alpha^2 t^2}{1-2\beta t-F_{\Phi,1}} $$ from which taking into account that $F_{\Phi,1}(0)=0$ the formula for $q=F_{\Phi,1}-1$ follows.

To prove the second formula note that the weighted graph $\Psi$ is isomorphic to the graph obtained by connecting the vertex $-1$ of a copy of $\Psi$ by two edges with weights $\alpha$ and $\beta$ to the one-vertex graph $\delta$. It follows that $$F_{\Psi,-1}=(\alpha+\beta)^2t^2P_{\Psi,-1}=\frac{(\alpha+\beta)^2t^2}{1-F_{\Psi,-1}}.$$ Solving the latter equation for $F_{\Psi,-1}$ we prove the formula for $p=F_{\Psi,-1}-1$.
\end{proof}

Observe that $x(0)=1$ and $p(0)=-1$ and thus the formulas $$\frac{1}{x(t)}\;\;\text{and}\;\;\frac{\alpha^2t^2}{p(t)+\beta^2t^2x(t)}$$
 define germs of analytic functions near the origin (\ie elements of $\mathbb C\{t\}$).

\begin{Prop}\label{PropCubic} In the ring $\mathbb C\{t\}$ one has:
\begin{equation}\label{EqMain}q+\frac{1}{x}+\beta^2t^2x-\frac{\alpha^2 t^2}{p+\beta^2 t^2x}=0.
\end{equation} Equivalently,
\begin{equation}\label{EqCubic}\beta^4 t^4x^3+(p+q)\beta^2t^2x^2+(pq+(\beta^2-\alpha^2)t^2)x+p=0.
\end{equation}
\end{Prop}
\begin{proof} Using Lemmas \ref{LmGraphSum} and \ref{LmDeltaEq} we obtain:
\begin{align*}F_{\Delta,v}=F_{\Phi_1\star\Delta_v,1}+\alpha^2t^2P_{\Psi_{-1}\star\Delta_v,-1}=
F_{\Phi,1}+\beta^2t^2P_{\Delta,v}+\frac{\alpha^2 t^2}{1-F_{\Psi_{-1}\star\Delta_v,-1}},\\
F_{\Psi_{-1}\star\Delta_v,-1}=F_{\Psi,-1}+\beta^2t^2P_{\Delta,v}.
\end{align*} Setting $x=P_{\Delta,v}$ using Lemma \ref{LmPF} we arrive at:
$$\frac{x-1}{x}=q+1+\beta^2t^2x-\frac{\alpha^2 t^2}{p+\beta^2 t^2x},$$ which is equivalent to the equation \eqref{EqMain}.
\end{proof}

\begin{Co}\label{CoGenGa} The following holds:
$$F_{\Upsilon,1/2}=1+2\beta t-\beta^2t^2x-\tfrac{1}{x}.$$
\end{Co}
\begin{proof} By Corollary \ref{PropGa1/2} and Lemmas \ref{LmPF} and \ref{LmGraphSum} one has:
\begin{align*}F_{\Upsilon,1/2}=F_{\tilde\Phi,1}+\frac{\alpha^2t^2}{1-F_{\Psi_{-1}\star\Delta_v,-1}},\\
%2\beta t+q+1+\end{align*}
F_{\Psi_{-1}\star\Delta_v,-1}=F_{\Psi,-1}+\beta^2t^2P_{\Delta,v}.\end{align*}
Combining these equations using \eqref{EqTilPhi} and \eqref{EqMain} we obtain the desired.
\end{proof}
\section{Proofs of Theorem \ref{ThMain} and Proposition \ref{Co-11}.}\label{SectionMain}
\subsection{Cauchy-Stieltjes transform.}\label{SubsecStiltjes}
Let $\mu$ be a finite measure on a segment $I=[s_1,s_2]\subset\mathbb R$. Its Cauchy-Stieltjes transform is defined by
$$S_\mu(z)=\int\limits_I \frac{\dd\mu(y)}{y-z},\;\; z\in\mathbb C\setminus\supp(\mu).$$ It is well-known that the function $S_\mu(z)$ is complex analytic on $\mathbb C\setminus \supp(\mu)$. Since for $y\in\mathbb R$ and $z\in\mathbb C\setminus \mathbb R$ the sign of $\im\tfrac{1}{y-z}$ coincides with the sign of $\im z$ the definition of Cauchy-Stieltjes transform implies that
\begin{equation}\im S_\mu(z)>0\;\;\text{if}\;\;\im z>0,\;\;\im S_\mu(z)<0\;\;\text{if}\;\;\im z<0.\label{EqSmuHalfplanes}
\end{equation}Introduce the moments and the generating function of the measure $\mu$ by:
$$m_k=\int\limits_I y^k\dd\mu(y),\;\;k\geqslant 0,\;\;R_\mu(t)=\sum\limits_{k\geqslant 0} m_kt^k.$$
Observe that \begin{equation}\label{EqStiltGen} S_\mu(z)=-z^{-1}R_\mu(z^{-1})\;\;\text{if}\;\; |z|>\max\{|s_1|,|s_2|\}.
\end{equation}

Recall that $\lambda$ stands for the Lebesgue measure on $\mathbb R$. For a finite measure $\mu$ on $\mathbb R$ denote by $G_\mu$ its distribution function:
\begin{equation}\label{EqDistributionFunction}
G_\mu(s)=\mu((-\infty,s]),\;\;s\in\mathbb R.
\end{equation}
Our computation of spectral measures relies on the following result from \cite{SilversteinChoi-Stieltjes-95}:
\begin{Th}\label{ThStilt}[Silverstein-Choi] Suppose that for some $s\in I$ the limit $S_+(s)=\lim\limits_{z\to s,\im z>0}\im S_\mu(z)$ exists. Then the distribution function $G_\mu$ is differentiable at $s$ and $G_\mu'(s)=\frac{1}{\pi}S_+(s)$.
\end{Th}
\noindent
Theorem \ref{ThStilt} implies the following:
\begin{Co}\label{CoSuppMu} Let $s\in\mathbb R$. Then $s\notin\supp(\mu)$ if and only if there exists an open segment $J$ containing $s$ such that $$\lim\limits_{z\to z_0,\im z>0}\im S_\mu(z)=0$$ for every $z_0\in J$.
\end{Co}
\begin{proof} Indeed, if such segment $J$ exists, then by Theorem \ref{ThStilt} $G_\mu$ is constant on $J$, therefore $\mu(J)=0$ and $s\notin\supp(\mu)$.
On the other hand, if $s\notin\supp(\mu)$ then there exists an open segment $J$ containing $s$ which does not intersect $\supp(\mu)$.  By definition of Cauchy-Stieltjes transform $S_\mu(z)$ is well-defined, analytic, and has real values on $J$. It follows that $S_+(z)=0$ on $J$.
\end{proof}
Though the domain of definition of $S_\mu(z)$ is $\mathbb C\setminus \supp(\mu)$ it is possible in some cases to continue analytically $S_\mu$ to neighborhoods of points from $\supp(\mu)$. To study these cases let us introduce an auxiliary
\begin{Def}\label{DefRealPole} We will say that $S_\mu$ has an isolated singularity at a point $z_0\in\supp(\mu)$ if it admits an analytic continuation from $\mathbb H_+=\{z\in\mathbb C:\re(z)>0\}$ to a punctured neighborhood of $z_0$. In other words, there exists an open neighborhood $U\subset \mathbb C$ of $z_0$ and an analytic function $\phi$ on $U\setminus \{z_0\}$ such that $\phi(z)=S_\mu(z)$ for all $z\in U\cap \mathbb H_+$.

If $z_0$ is a pole  of $\phi$ we will say that $S_\mu$ has a pole at $z_0$.

 If $\phi(z)\in\mathbb R$ for all $z\in (U\cap \mathbb R)\setminus \{z_0\}$ we will say that $S_\mu$ has a real isolated singularity at $z_0$.
\end{Def}
\begin{Rem}\label{RemConj} Since $S_\mu$ is real on the real line one has $S_\mu(\bar z)=\overline{(S_\mu(z))}$ for all $z\notin I$. Moreover, \eqref{EqSmuHalfplanes} implies that all zeros of $S_\mu$ on $\mathbb R\setminus \supp(\mu)$ are simple and all real isolated singularities of $S_\mu$ on $\supp(\mu)$ are simple poles.
\end{Rem}

Now, let $\Gamma$ be a weighted graph of bounded degree with  a uniformly bounded weight, let $v$ be a vertex of $\Gamma$, and let $\mu_{\Gamma,v}$ be the correpsonding spectral measure of $H_\Gamma$ (see \eqref{EqHeckeGamma}). Then the $n$th moment $m_n$ of $\mu_{\Gamma,v}$ is equal to $p^{(n)}_{\Gamma,v}$ for every $n\in\mathbb Z_+$ (see \eqref{EqPGammavnDef} and \eqref{EqpgvHG}). Therefore, $R_\mu(t)=P_{\Gamma,v}(t)$. In particular, we have:
\begin{equation}\label{EqSmuPmu} S_{\mu_{\Gamma,v}}(z)=-z^{-1}P_{\Gamma,v}(z^{-1})\;\;\text{if}\;\; |z|>\|H_\Gamma\|.
\end{equation}
\subsection{The spectral measure $\mu_{\Delta,v}$ of $H_\Delta$.}\label{SubsetKestenHDelta}
Formulas \eqref{EqqpForm} define analytic functions $p(t),q(t)$ near zero which admit analytic continuations to the lower half-plane $\mathbb H_-=\{z\in\mathbb C:\im z<0\}$. We then extend these functions by continuity onto the real line. We denote the obtained  functions on $\overline{\mathbb H}_-$ by  $P(t)$ and $Q(t)$ correspondingly.

Recall that $x(t)$ (see Subsection \ref{EqqpDef}) is a convergent power series in $t$ near $0$. From \eqref{EqStiltGen} we deduce that $x$ admits an analytic continuation to $\mathbb H_-$. Denote the obtained function on $\mathbb H_-$ by $X(t)$. Using \eqref{EqCubic} and Cardano formulas we obtain that $X(t)$ extends to a continuous function on $\overline{\mathbb H_-}$. Moreover, \begin{equation}\label{EqCubic1}\beta^4 t^4X^3(t)+(P(t)+Q(t))\beta^2t^2X^2(t)+(P(t)Q(t)+(\beta^2-\alpha^2)t^2)X(t)+P(t)=0
\end{equation} for all $t\in\overline{\mathbb H_-}$.

Further, let \begin{equation}\label{EqPQ*}P_*(z)=-zP(z^{-1}),
\;\;Q_*(z)=-zQ(z^{-1}),\;\;V(z)=-z^{-1}X(z^{-1})\end{equation} for $z\in\overline{\mathbb H}_+\setminus \{0\}$.
\begin{Rem}\label{RemMuDel} By \eqref{EqStiltGen} one has
$$V(z)=S_{\mu_\Delta}(z)\;\;\text{on}\;\;\mathbb H_+,$$ where $\mu_\Delta=\mu_{\Delta,v}$ is the spectral measure of $H_\Delta$ corresponding to the vector $\delta_v$. In particular, $\im V(z)>0$ for $z\in\mathbb H_+$.
\end{Rem}
Observe that for any $z\in\overline{\mathbb H}_+\setminus \{0\}$ one has:
\begin{equation}\label{EqPQFormulas} P_*(z)=\tfrac{z}{2}\pm\tfrac{1}{2}\sqrt{z^2-4(\alpha+\beta)^2},\;\;
Q_*(z)=\tfrac{z}{2}+\beta\pm\tfrac{1}{2}\sqrt{z^2-4\beta z+4(\beta^2-\alpha^2)}.
\end{equation}
It is not hard to see that $P_*$ and $Q_*$ extend by continuity to $0$.

From \eqref{EqCubic1} we obtain that for all $z$ from $\overline{\mathbb H}_+\setminus\{0\}$ the value $V(z)$ is a solution of the following equation:
\begin{equation}\label{EqV}\beta^4v^3+(P_*(z)+Q_*(z))\beta^2v^2+
(P_*(z)Q_*(z)+(\beta^2-\alpha^2))v+P_*(z)=0.\end{equation}
If $\beta$ is nonzero  then using Cardano formulas we obtain that  $V(z)$ extends by continuity to $0$. If $\beta=0$ and $\alpha\neq 0$ then %from \eqref{EqV} we obtain that
$$P_*(z)=Q_*(z)=\tfrac{z}{2}\pm\tfrac{1}{2}\sqrt{z^2-4\alpha^2},\;\;\text{and so}\;\;V(z)=\frac{P_*(z)}{\alpha^2-P_*^2(z)}$$ also extends by continuity to 0.
Using Theorem \ref{ThStilt} we obtain:
\begin{Co}\label{CoKestenHDelta} The spectral measure $\mu_\Delta$ is absolutely continuous with respect to the Lebesgue measure $\lambda$. Moreover,
$$\frac{\dd\mu_\Delta(z)}{\dd \lambda(z)}=\tfrac{1}{\pi}\im V(z)\;\;\text{for}\;\;z\in\mathbb R.$$
\end{Co}
\begin{Lm}\label{LmMain1} For $z\in\overline{\mathbb H_+}$ unless  $z=2(\alpha+\beta)=0$ one has $$P_*(z)\neq 0,\;V(z)\neq 0,\;P_*(z)+\beta^2 V(z)\neq 0$$ and, moreover, $V(z)$ is a solution of the equation \begin{equation}\label{EqV1}Q_*(z)+\frac{1}{v}+\beta^2v-\frac{\alpha^2}{P_*(z)+\beta^2 v}=0.
\end{equation}
\end{Lm}
\begin{proof} Assume $z\in\overline{\mathbb H_+}$ and either $z\neq 0$ or $z\neq 2(\alpha+\beta)$. Then $P_*(z)\neq 0$ by \eqref{EqqpForm}, \eqref{EqPQ*} and \eqref{EqPQFormulas}. Consequently, $V(z)\neq 0$ by \eqref{EqV}. Now, since $\alpha^2V(z)\neq 0$ from \eqref{EqV} we derive that $P_*(z)+\beta^2 V(z)\neq 0$. Formula \eqref{EqV1} follows from \eqref{EqV}.
\end{proof}
\subsection{The spectral measure $\mu_{m,1/2}$ of $H_m$.}
As before assume that $\alpha,\beta\in \mathbb R\setminus \{0\}$. To simplify the notations set $y=P_{\Upsilon,1/2}$, where as before $\Upsilon$ is the Schreier graph of the action of the Thompson group $F$ on the orbit of $1/2$. Then $y(t)$ is a convergent series near the origin. Set $W(z)=-z^{-1}y(z^{-1})$. By \eqref{EqSmuPmu} one has $W(z)=S_{\mu_m}(z)$, where $\mu_m=\mu_{m,1/2}$ is the spectral measure of the operator $H_m$ corresponding to the vector $\delta_{\frac{1}{2}}$.
\begin{Prop}\label{PropUps}
$W(z)$ extends  to a continuous function on $\overline{\mathbb H}_+$ analytic on $\mathbb H_+$. One has:
\begin{equation}\label{EqW} W(z)=\frac{V(z)}{(\beta V(z)+1)^2}.
\end{equation} For $z\in\mathbb R$ $\im W(z)>0$ if and only if $\im V(z)>0$.
\end{Prop}
\begin{proof} From Lemma \ref{LmPF} and Corollary \ref{CoGenGa} we obtain:
$$y=P_{\Upsilon,1/2}=\frac{1}{1-F_{\Upsilon,1/2}}=\frac{1}{\beta^2t^2x+\frac{1}{x}-2\beta t}.$$ Substituting $W(z)=-z^{-1}y(z^{-1})$ and $V(z)=-z^{-1}x(z^{-1})$ in the latter formula we obtain formula \eqref{EqW}.

If $\beta V(z_0)+1=0$ for some $z_0\in\mathbb R$ then $W(z)$ has a pole of order at least 2 at $z_0$. This contradicts to Remark \ref{RemConj} and shows that $\beta V(z)+1\neq 0$ for all $z\in \mathbb R$.

Further, from \eqref{EqW} we obtain that $W(z)$ is real whenever $V(z)$ is real. From \eqref{EqV1} we obtain:
$$\frac{1}{W(z)}=-Q_*(z)+2\beta+\frac{\alpha^2}{P_*(z)+\beta^2 V(z)}.$$ Since $\im P_*(z)\geqslant 0$ and $\im Q_*(z)\geqslant 0$ for $z\in\overline{\mathbb H}_+$ we obtain that $\im W(z)>0$ whenever $\im V(z)>0$. This finishes the proof.
\end{proof}
\subsection{Spectrum of $H_\Delta$}

\begin{Lm}\label{LmImag} Assume that one of the following two conditions hold:
 \begin{itemize}
\item[$1)$] $z\in\mathbb H_+$;
\item[$2)$] $z\in\mathbb R$ such that $|\alpha+\beta|>\tfrac{|z|}{2}$ or $|\alpha |>|\tfrac{z}{2}-\beta|$,
\end{itemize} then  the equation \eqref{EqV} has no real solutions.
\end{Lm}
\begin{proof}
From Lemma \ref{LmqpForm} and \eqref{EqPQ*} we deduce that $\im P_*(z)> 0$ and $\im Q_*(z)> 0$ on $\mathbb H_+$. In addition, for $z\in\mathbb R$ one has
\begin{itemize}
\item{} if $|\alpha+\beta|>\tfrac{|z|}{2}$ then $\im P_*(z)>0$;
\item{} if $|\alpha |>|\tfrac{z}{2}-\beta|$ then $\im Q_*(z)>0$.
\end{itemize}
 Let $v\in\mathbb R$ be a solution of \eqref{EqV}. If $z$ satisfies at least one of the conditions of Proposition \ref{LmImag} then $v\neq 0$ and $P_*(z)+\beta^2v\neq 0$ (see the proof of Lemma \ref{LmMain1}).  Therefore \eqref{EqV1} is true and so
$$\im\bigg(Q_*(z)-\frac{\alpha^2}{P_*(z)+\beta^2 v}\bigg)=0.$$
This contradicts to the fact that $\im P_*(z)\geqslant 0$, $\im Q_*(z)\geqslant 0$ and at least one of the inequalities is strict.
\end{proof}
\noindent
Recall that $\Delta$ is the subgraph of the Schreier graph of the action of $F$ on the orbit of $1/2$ defined in Lemma \ref{LmDeltaEq} and illustrated on Figure \ref{FigSchreierThompsonSelfsim}.
\begin{Prop}\label{PropSameSign} If $\alpha,\beta\neq 0$ are of the same sign then the support of the spectral measure $\mu_\Delta=\mu_{\Delta,v}$ of the Laplace type operator $H_\Delta$ (see (\ref{EqHeckeGamma})) is $[-2|\alpha+\beta|,2|\alpha+\beta|]$ and coincides with the spectrum of $H_\Delta$.
\end{Prop}
\begin{proof} Without loss of generality assume that $\alpha,\beta>0$. Lemma \ref{LmImag}, Corollary \ref{CoSuppMu} and Remark \ref{RemMuDel} imply that $\sigma(H_\Delta)\supset [-2(\alpha+\beta),2(\alpha+\beta)]$. On the other hand, since for every vertex of $\Delta$ the sum of the weights of adjacent edges is at most $2(\alpha+\beta)$ we have $\|H_\Delta\|\leqslant 2(\alpha+\beta)$. It follows that $\sigma(H_\Delta)\subset [-2(\alpha+\beta),2(\alpha+\beta)]$.
\end{proof}
\begin{Prop}\label{PropVUniqPos} For any $\alpha,\beta\in\mathbb R$ at least one of which is nonzero and any $z\in \overline{\mathbb H}_+$ the equation \eqref{EqV} has at most one solution $v=v(z)$ (counting multiplicity) such that $\im (v)>0$. For $z\in\mathbb H_+$ such solution is exactly one and coincides with $V(z)$. If $z\in \mathbb R$ and such solution exists then it coincides with $V(z)$.
\end{Prop}
\begin{proof}  If $\beta=0,\alpha\neq 0$ then the equation \eqref{EqV} has only one solution. Let $\beta\neq 0$. Without loss of generality we may assume that $\beta=1$. Let $z_0$ be such that \eqref{EqV} for $z=z_0$ has two solutions (counting multiplicity) with positive imaginary part.
If $z_0\in\mathbb R$ using either Cardano's formulas or Implicit Function Theorem we obtain that one can find near $z_0$ a point $z_1\in\mathbb H_+$ satisfying the same condition. Therefore, without loss of generality we may assume that $z_0\in\mathbb H_+$. Since the coefficients of \eqref{EqV} depend continuously on $\alpha$ Cardano's formulas together with Lemma \ref{LmImag} imply that for $\alpha=0$ and $z=z_0$ \eqref{EqV} also have two solutions with positive imaginary part. Finally, same arguments show that for $z=0$ there are also two solutions with positive imaginary part. But for $\alpha=0,\beta=1$ one has $P_*(0)=i,Q_*(0)=0$ and the equation \eqref{EqV} for $z=0$ has the following form:
$$v^3+iv^2+v+i=0.$$ The latter equation has one simple solution $v=i$ with positive imaginary part. This contradiction together with Remark \ref{RemMuDel} finishes the proof.
\end{proof}

Recall that a cubic equation with real coefficients has imaginary roots if and only its discriminant is negative. Denote by $D(z)$ the discriminant of the equation \eqref{EqV}:
\begin{equation}\label{EqDiscr} D(z)=18abcd-4b^3d+b^2c^2-4ac^3-27a^2d^2,
\end{equation} where $a=\beta^4,b=\beta^2(P_*(z)+Q_*(z)),c=P_*(z)Q_*(z)+\beta^2-\alpha^2,d=P_*(z)$. Using Corollary \ref{CoSuppMu} and Remark \ref{RemMuDel}, Lemma \ref{LmImag} and Proposition \ref{PropVUniqPos} we obtain that
\begin{equation}\label{EqSpecDis}\supp(\mu_\Delta)=\overline{\{z\in\mathbb R:|\alpha+\beta|>\tfrac{|z|}{2}\;\text{or}\;|\alpha |>|\tfrac{z}{2}-\beta|\;\text{or}\;D(z)>0\}}.
\end{equation}
Let us prove an auxiliary lemma.
\begin{Lm}\label{LmRootsIrrat} Let $U\subset C$ be a domain, $p_1(z),\ldots,p_n(z)$ be complex polynomials and $s_1(z),\ldots s_n(z)$ be analytic functions on $U$ such that $s_j^2(z)=p_j(z)$, $j=1,\ldots,n$, for all $z\in U$. Let $R$ be a polynomial in $n+1$ variable. Then the equation
\begin{equation}\label{EqR}R(z,s_1(z),\ldots,s_n(z))=0\end{equation} either holds for all $z\in U$ or has only finitely many solutions on $U$.
\end{Lm}
\begin{proof}
The equation \eqref{EqR} can be written in the form
$$ s_n(z)R_1(z,s_1(z),\ldots,s_{n-1}(z))=R_2(z,s_1(z),\ldots,s_{n-1}(z)),$$ where $R_1,R_2$ are polynomials of $n$ variables. Thus, if $z$ is a solution of \eqref{EqR} we have:
$$p_n^2(z)R_1^2(z,s_1(z),\ldots,s_{n-1}(z))-R_2^2(z,s_1(z),\ldots,s_{n-1}(z))=0.$$
From the latter one can complete the proof by induction on $n$.
\end{proof}

Let $I$ be one of the open intervals (finite or infinite) on which the points $-2(\alpha+\beta),2(\alpha+\beta),2(\beta-\alpha)$ divide $\mathbb R$. Let $U$ be an open simply-connected set containing $I$ but not containing any of the points $-2(\alpha+\beta),2(\alpha+\beta),2(\beta-\alpha)$. From Lemma \ref{LmRootsIrrat} we obtain that $D(z)=0$ either $a)$ has only finitely many solutions on $U$ or $b)$ is true for all $z\in U$. In case $a)$ set $S_{\alpha,\beta}=\{z\in\mathbb R:D(z)=0\}$. By analyticity and continuity arguments the case $b)$ would imply that $D(z)=0$ on $\overline{\mathbb H_+}$. In this case we set $S_{\alpha,\beta}=\varnothing$.

Using Theorem \ref{ThStilt} we deduce the following:
\begin{Prop}\label{PropFiniteInt} For any $\alpha,\beta\in \mathbb R\setminus\{0\}$ the set $\supp(\mu_\Delta)$ is a union of at most finite number of closed intervals with the end-points belonging to $S_{\alpha,\beta}\cup\{-2(\alpha+\beta),2(\alpha+\beta),2(\beta-\alpha)\}$.
\end{Prop}
\subsection{Proof of Theorem \ref{ThMain}.}
By \eqref{EqStiltGen} one has $S_{\mu_m}(z)=W(z)$ on $\mathbb H_+$. Using Theorem \ref{ThStilt}, Proposition \ref{PropUps} and Corollary \ref{CoSuppMu} we obtain that
\begin{equation}\label{EqSuppMDelta}\supp(\mu_m)=\overline{\{z\in\mathbb R:\im W(z)>0\}}=\overline{\{z\in\mathbb R:\im V(z)>0\}}=\supp(\mu_\Delta).\end{equation} Moreover, $\mu_m$ is absolutely continuous with respect to the Lebesgue measure $\lambda$ on $\mathbb R$ and \eqref{EqDmuDlam} is true. Finally, from Proposition \ref{PropFiniteInt} we obtain that $\supp(\mu_m)$ is a union of a finite number of closed intervals.
\subsection{Proof of Proposition \ref{Co-11}.}
Let $\alpha=-1,\beta=1$. Lemma \ref{LmImag} and Theorem \ref{ThStilt} imply that $\supp(\mu_\Delta)\supset [0,4]$. Observe that for $z\notin [0,4]$ the coefficients of the cubic equation \eqref{EqV} are real. Clearly, $\supp(\mu_\Delta)\subset [-4,4]$. For $z\in[-4,0)$ the discriminant $D(z)$ of \eqref{EqCubic1} (see \eqref{EqDiscr}) is of the form:
$$D(z)=-4z-14z^2+2z^4+\sqrt{z(z-4)}(6z-4z^2-2z^3).$$ Algebraic transformations show that the set of zeros of $D(z)$ on $[-4,0)$ is a subset of the set of roots of $z^4-2z^2+16z+1$. The latter equation  has two real roots $z_1,z_2$ defined in the formulation of Proposition \ref{Co-11}. One can check that $z_1,z_2$ are zeros of $D(z)$. Moreover on $(-4,0)\setminus [z_1,z_2]$ the function $D(z)$ is positive and so \eqref{EqV} has three real solutions. On $(z_1,z_2)$ the function $D(z)$ is negative and so \eqref{EqV} has one real and two complex conjugate solutions with nonzero imaginary part. Using Proposition \ref{PropVUniqPos} we obtain that $\supp(\mu_\Delta)=[z_1,z_2]\cup [0,4]$.
 Finally, using \eqref{EqSuppMDelta} we obtain Proposition \ref{Co-11}.

\section{On simple random walk on $\Upsilon$: proofs of Theorem \ref{ThAsympMoments} and Proposition \ref{PropIntDens}.}\label{SectionAsymp}
Let $\alpha=\beta=1/4$ so that the operator $H_{m,1/2}$ is the Markov operator associated to the simple random walk on $\Upsilon=\Gamma_{1/2}$. The moment generating function for the operator $H_{m,1/2}$ is $P_{\Upsilon,1/2}$. Recall that by Lemma \ref{LmPF} and Corollary \ref{CoGenGa} we have:
\begin{equation}\label{EqPUFUX}P_{\Upsilon,1/2}=\frac{1}{1-F_{\Upsilon,1/2}},\;\;
F_{\Upsilon,1/2}=
1+\tfrac{t}{2}-\tfrac{1}{16}t^2x^2-\tfrac{1}{x},\end{equation}
where $x(t)$ is a solution of the following equation (see Proposition \ref{PropCubic} and Lemma \ref{LmqpForm}):
\begin{equation}\label{EqXFor1/4}\tfrac{1}{256}t^4x^3+\tfrac{1}{16}(p+q)t^2x^2+pqx+p=0.\end{equation} By \eqref{EqqpForm} one has $$q(t)=-\tfrac{1}{2}-\tfrac{t}{4}-\tfrac{1}{2}\sqrt{1-t},\;\;
p(t)=-\tfrac{1}{2}-\tfrac{1}{2}\sqrt{1-t^2}.$$

To study the {\bf behavior of $x(t)$ near $t=1$} we substitute $t=1-s^2$. We obtain for $t\leqslant 1$ and $s\geqslant 0$:
$$q(t)=-\frac{3}{4}+\frac{1}{4}s^2-\frac{1}{2}s=
-\frac{3}{4}-\frac{1}{2}s+O(s^2),\;\;
p(t)=-\frac{1}{2}-\frac{1}{2}s\sqrt{2-s^2}=
-\frac{1}{2}-\frac{\sqrt{2}}{2}s+O(s^3).$$ Let $u(s)=x(1-s^2)$. The equation \eqref{EqXFor1/4} can be rewritten as:
\begin{align*}F(u,s):=a_0(s)u^3+a_1(s)u^2+a_2(s)u+a_3(s)=0,\\
a_0(s)=\frac{1}{256}+O(s^2),\;\;
a_1(s)=-\frac{5}{64}-\frac{(\sqrt{2}+1)s}{32}+O(s^2),\\
a_2(s)=\frac{3}{8}+\frac{(2+3\sqrt{2})s}{8}+O(s^2),\;\;
a_3(s)=
-\frac{1}{2}-\frac{\sqrt 2 s}{2}+O(s^2),\end{align*} where $a_i(s)$ are analytic near $s=0$. The solutions of the cubic equation $F(u,0)=0$ are:
$$u_1=4(2+\sqrt 2),\;\;u_2=4,\;\;u_3=4(2-\sqrt 2).$$
One has:
$$\frac{\partial F}{\partial u}|_{s=0}=\frac{3}{256}u^2-\frac{5}{32}u+\frac{3}{8}.$$ The latter expression is nonzero at $u=u_j$ for $j=1,2,3$. We also calculate:
$$\frac{\partial F}{\partial s}|_{s=0}=-\frac{1+\sqrt 2}{32}u^2+\frac{2+3\sqrt 2}{8}u-\frac{\sqrt 2}{2}.$$ By Implicit Function Theorem we have that for each $j=1,2,3$ the equation $F(u,s)=0$ has an analytic solution $U_j(s)$ near $s=0$ such that $U_j(0)=u_j$. Moreover,
$$U_j'(0)=-\frac{\partial F}{\partial s}(u_j,0)/\frac{\partial F}{\partial u}(u_j,0)=\frac{\frac{1+\sqrt 2}{32}u_j^2-\frac{2+3\sqrt 2}{8}u_j+\frac{\sqrt 2}{2}}{\frac{3}{256}u_j^2-\frac{5}{32}u_j+\frac{3}{8}}.$$
Substituting the values of $u_j$ in the above formula we obtain:
$$U_1'(0)\approx 9.657,\;\;U_2'(0)\approx 19.314,\;\;U_3'(0)\approx -9.657.$$

Recall that the coefficients of $x(t)$ are the return probabilities of a random walk, and so are nonnegative. It follows that $x'(t)\geqslant 0$ for $0\leqslant t<1$. We obtain that $u'(s)=-2sx'(t)\leqslant 0$ for $1>s>0$. Therefore, $u(s)=U_3(s)$. Thus, $x(1-s^2)=u(s)=C_1+C_2s+O(s^2)$, where $C_1=4(2-\sqrt 2), C_2\approx -9.657$. We arrive at $x(t)=C_1+C_2\sqrt{1-t}+O(1-t)$ for $t$ near $1$.

Similarly, to study {\bf the behavior of $x(t)$ near $t=-1$} we substitute $t=s^2-1$, $s\geqslant 0$. We have
\begin{align*}p(t)=-\frac{1}{2}-\frac{1}{2}s\sqrt{2-s^2}=
-\frac{1}{2}-\frac{\sqrt 2}{2}s+O(s^2),\\
q(t)=-\frac{1}{2}-\frac{s^2-1}{4}-\frac{1}{2}\sqrt{2-s^2}=
-\frac{1}{4}-\frac{\sqrt 2}{2}+O(s^2).\end{align*}
Now, let $u(s)=x(s^2-1)$. The equation \eqref{EqXFor1/4} becomes
\begin{equation}\label{EqXNear-1}\begin{split}
F(u,s):=b_0(s)u^3+b_1(s)u^2+b_2(s)u+b_3(s)=0,\;\;\text{where}\\ b_0(s)=\frac{1}{256}+O(s^2),\;\;b_1(s)=-\frac{3+2\sqrt 2}{64}-\frac{\sqrt 2}{32}s+O(s^2),\\
b_2(s)=\frac{1+2\sqrt{2}}{8}+\frac{4+\sqrt 2}{8}s+O(s^2),\;\;b_3(s)=-\frac{1}{2}-\frac{\sqrt{2}}{2}s+O(s^2),
\end{split}\end{equation} and $b_1(s),b_2(s),b_3(s)$ are analytic near $s=0$.

The equation \eqref{EqXNear-1} for $s=0$ has the following form:
\begin{equation}\label{EqCubicVat-1} \frac{1}{256}u^3-\frac{3+2\sqrt 2}{64}u^2+\frac{1+2\sqrt 2}{8}u-\frac{1}{2}=0.
\end{equation}
The roots are:
\begin{equation}\label{EqUjat-1}\begin{split}u_1=6+2\sqrt 2-2\sqrt{11+2\sqrt 2}\approx 1.3911,\;\;u_2=4\sqrt{2}\approx 5.657,\\u_3=6+2\sqrt 2+2\sqrt{11+2\sqrt 2}\approx 16.267.\end{split}\end{equation}
One has:
$$\frac{\partial F}{\partial u}|_{s=0}=\frac{3}{256}u^2-\frac{3+2\sqrt 2}{32}u+\frac{1+2\sqrt 2}{8}.$$ The latter expression is nonzero at $u_1,u_2,u_3$. We have:
$$\frac{\partial F}{\partial s}|_{s=0}=-\frac{\sqrt 2}{32}u^2+\frac{4+\sqrt 2}{8}u-\frac{\sqrt 2}{2}.$$
By Implicit Function Theorem we have that for each $j=1,2,3$ the equation $F(u,s)=0$ has an analytic solution $U_j(s)$ near $s=0$ such that $U_j(0)=u_j$. We have $x(s^2-1)=U_j(s)$ for some $j$ for $s\geqslant 0$ sufficiently close to zero. Since the Taylor coefficients of $x(t)$ are nonnegative, we have $u_j=x(-1)<x(1)=4(2-\sqrt 2)\approx 2.343$. Comparing the latter with \eqref{EqUjat-1} we obtain that $j=1$ and $x(s^2-1)=U_1(s)$. Further,
$$U_1'(0)=-\frac{\partial F}{\partial s}(u_1,0)/\frac{\partial F}{\partial u}(u_1,0)=\frac{\frac{\sqrt 2}{32}u_1^2-\frac{4+\sqrt 2}{8}u_1+\frac{\sqrt 2}{2}}{\frac{3}{256}u_1^2-\frac{3+2\sqrt 2}{32}u_1+\frac{1+2\sqrt 2}{8}}\approx -0.6.$$
In particular, $U_1'(0)\approx -0.6$.
We obtain that $x(t)=C_3+C_4\sqrt{1+t}+O(1+t)$ for $t$ near $-1$ with $C_3=6+2\sqrt 2-2\sqrt{11+2\sqrt 2}\approx 1.3911,C_4\approx -0.6$.

Now, using formulas \eqref{EqPUFUX} we obtain, that $P_{\Upsilon,1/2}(t)$ has near $\pm 1$ the asymptotical behavior similar to the asymptotical behavior of $x(t)$:
\begin{align}\label{EqPUpsilonAsymptotics}\begin{split}
P_{\Upsilon,1/2}(t)=\widetilde C_1+\widetilde C_2\sqrt{1-t}+O(1-t)\;\;\text{near}
\;\;t=1,\\ \;\;P_{\Upsilon,1/2}(t)=\widetilde C_3+\widetilde C_4\sqrt{1+t}+O(1+t)
\;\;\text{near}\;\;t=-1\end{split}\end{align} for some constants $\widetilde C_j\in\mathbb R$, $1\leqslant j\leqslant 4$. Further, let $P_{\Upsilon,1/2}(t)=\sum\limits_{n=0}^\infty p_nt^n$. The function $$G(s)=\tfrac{1}{2}(P_{\Upsilon,1/2}(\sqrt s)+P_{\Upsilon,1/2}(-\sqrt s))=\sum\limits_{n=0}^\infty p_{2n}s^n$$ is analytic in a region of the form $D_\epsilon=\{s\in\mathbb C:|s|<1+\epsilon\}\setminus [1,1+\epsilon)$ for some $\epsilon>0$ and has the following asymptotical behavior near $s=1$:
$$G(s)=\tfrac{1}{2}(\widetilde C_1+\widetilde C_3)+\tfrac{1}{2}(\widetilde C_2+\widetilde C_4)\sqrt{1-s}+O(1-s).$$
Using results of Flajolet-Odlyzko \cite{FlajoletOdlyzko-Singularity-90} (see also Theorem of Section 16.8 in \cite{Woess-Asymptotic}) we obtain that
$p_{2n}=-\tfrac{1}{2}(\widetilde C_2+\widetilde C_4)n^{-3/2}+O(n^{-2})$ when $n\to\infty$. Similarly, we obtain that
$p_{2n+1}=-\tfrac{1}{2}(\widetilde C_2-\widetilde C_4)n^{-3/2}+O(n^{-2})$ when $n\to\infty$. This finishes the proof of Theorem \ref{ThAsympMoments}.

From the asymptotics \eqref{EqPUpsilonAsymptotics} of $P_{\Upsilon,1/2}(t)$ near $t=1$ and \eqref{EqSmuPmu} we obtain:
$$S_{\mu_{\Upsilon,1/2}}(z)=-\widetilde C_1-\widetilde C_2\sqrt{z-1}+O(z-1)$$ near $z=1$. Notice that \eqref{EqPUpsilonAsymptotics} and the fact that $P_{\Upsilon,1/2}(t)$ is increasing on $[0,1)$ implies that $\widetilde C_2<0$. Using Theorem \ref{ThStilt} we get
$$\frac{\dd\mu_{\Upsilon,1/2}(z) }{\dd\lambda(z)}=-\tfrac{1}{\pi}\widetilde C_2\sqrt{1-z}+O(1-z)$$ for $z<1$ near $z=1$. Integrating we arrive at:
$$\mu_{\Upsilon,1/2}(1-z,1]=-\tfrac{2}{3\pi}\widetilde C_2(1-z)^{\frac{3}{2}}+O(1-z)^2,$$ which finishes the proof of Proposition \ref{PropIntDens}.

\section{Inclusion of spectra and dependence of spectral measure $\mu_v$ on $v$.}\label{SectionMuvH}
Here we study the relation between the spectrum of the Laplace type operator and the supports of the spectral measures for a certain class of graphs including Schreier graphs of the action of $F$ on $[0,1]$.

Observe that for the simple random walk on a Cayley graph of a finitely generated countable group spectral measures corresponding to different vertices coincide and their support is equal to the spectrum of the Markov operator (see \cite{Kesten-RandomWalk-59}). However, for Laplace type operators on general graphs (or even on Schreier graphs) this is not always true. Below we provide two examples.

Let $\Gamma$ be a finite weighted graph with real symmetric weight. Consider the corresponding Laplace-type operator $L=H_\Gamma$. One can find the spectral measure $\mu_v=\mu_{\Gamma,v}$ corresponding to a vertex $v$ as follows. Since $L$ is self-adjoint, it has a spectral decomposition. Let $\lambda_i,\xi_i,i=1,\ldots ,|\Gamma|,$ be the eigenvalues and the eigenvectors of $L$, where $|\Gamma|$ is the number of vertices of $\Gamma$. We normalize $\xi_i$ to have $\|\xi_i\|=1$, $i=1,\ldots,|\Gamma|$. Let
$$\delta_v=\sum\alpha_i\xi_i$$ be the decomposition of $\delta_v$ over the basis $\{\xi_i\}$ of $l^2(\Gamma)$. Then the  spectral measure  $\mu_v$ is supported on the set $\{\lambda_i:1\leqslant i\leqslant n,\,\alpha_i\neq 0\}$ and \begin{equation}\label{Eqmuvdelta}\mu_v(\{\lambda_i\})=|\alpha_i|^2=
|(\delta_v,\xi_i)|^2\;\; \text{for}\;\;1\leqslant i\leqslant n.\end{equation}
\begin{Ex}\label{Ex5vertices} Let $\Gamma$ be the undirected graph with the vertex set $V=\{-2,-1,0,1,2\}$ and the edge set $E=\{(-2,-1),(-1,0),(0,1),(1,2)\}$. Assign weight $1$ to each edge. Then the corresponding Laplace-type operator can be written as a $5\times 5$ matrix with $1$ at each entry directly above or below the diagonal and $0$ every where else. The eigenvalues and the eigenvectors are:
\begin{align*}\lambda_1=0,\;\xi_1=\tfrac{1}{\sqrt 3}(1,0,-1,0,1),\;\lambda_2=1,
\;\xi_2=\tfrac{1}{\sqrt 3}(-1,-1,0,1,1),\\
\lambda_3=-1,\;\xi_3=\tfrac{1}{2}(-1,1,0,-1,1),\; \lambda_4=\sqrt{3},\;\xi_4=\tfrac{1}{2\sqrt 3}(1,\sqrt 3,2,\sqrt 3,1),\\ \lambda_5=-\sqrt{3},\;\xi_5=\tfrac{1}{2\sqrt 3}(1,-\sqrt{3},2,-\sqrt{3},1).\end{align*}
From \eqref{Eqmuvdelta} we obtain \begin{align*}\supp(\mu_{-2})=\supp(\mu_2)=\{\lambda_i:1\leqslant i\leqslant 5\},\\ \supp(\mu_{-1})=\supp(\mu_1)=\{\lambda_i:2\leqslant i\leqslant 5\},\;\supp(\mu_0)=\{\lambda_1,\lambda_4,\lambda_5\}.\end{align*}
\end{Ex}
\begin{figure}\centering\includegraphics[width=0.5\linewidth]{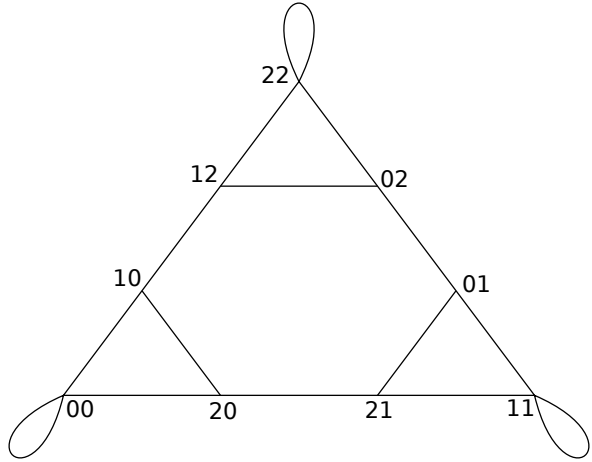}
\caption{The Schreier graph of the action of $H^{(3)}$ on $V_2$.}\label{FigHanoi} \end{figure}
\begin{Rem} As Pierre de la Harpe pointed to the authors, the graph with three vertices $V=\{-1,0,1\}$ and the edge set $E=\{\{-1,0\},\{0,1\}\}$ also has this property, that for some vertex (namely, the vertex $0$) the support of the spectral measure is a strict subset of the spectrum of the Markov operator.
\end{Rem}
\begin{Ex} Consider the Schreier graph $\Gamma$ of the action of the Hanoi Towers group $H^{(3)}$ on the second level $V_2\simeq \{0,1,2\}^2$ of the ternary rooted tree $T_3$ (see \cite{GrigorchukSunic-Hanoi-08}). This Schreier graph is shown on Figure \ref{FigHanoi}. Let $M$ be the Markov operator corresponding to the simple random walk on $\Gamma$.
One can show that $-2/3$ is a simple eigenvalue of $M$ with the eigenvector given by:
 $$\alpha_{00}=\alpha_{11}=\alpha_{22}=0,\;\;\alpha_{10}=\alpha_{02}=\alpha_{21}=1,\;\;
 \alpha_{20}=\alpha_{12}=\alpha_{01}=-1.$$
This eigenvector is orthogonal to $\delta_{00}$ but not orthogonal to $\delta_{10}$. Thus, by \eqref{Eqmuvdelta} $-2/3$ is contained in the support of $\mu_{10}$ but not in the support of $\mu_{00}$.
\end{Ex}

\begin{Def} Let us call a connected graph $\Gamma$ \emph{tree-like} if for every pair of adjacent vertices $v\neq w$ the graph obtained by removing from $\Gamma$ all the edges joining $v$ and $w$ is disconnected. \end{Def}
\noindent Note that a connected graph is tree-like if and only if after removing all loops and replacing multiple edges by single edges it becomes a tree. It is not hard to see that the Schreier graphs of the action of $F$ on $[0,1]$ are tree-like.

\begin{Lm}\label{LmAtomPoleEigenvalue} Let $\Gamma$ be a weighted graph of bounded degree with real symmetric uniformly bounded weight. Let $a\in\mathbb R$ and let $v$ be a vertex of $\,\Gamma$. Then $a$ is an isolated atom of $\mu_{\Gamma,v}$ if and only if
 $a$ is a real pole of $S_{\mu_{\Gamma,v}}$ (per Definition \ref{DefRealPole}). Moreover, if the above equivalent conditions hold then $a$ is an eigenvalue of $H_\Gamma$.
\end{Lm}
\begin{proof}  Assume that $a$ is an isolated atom of $\mu:=\mu_{\Gamma,v}$. Using the formula for $S_\mu$:
$$S_{\mu}(z)=\int\limits_{\mathbb R} \frac{\dd\mu(y)}{y-z}$$ (see Section \ref{SubsecStiltjes} for details) we immediately obtain that $a$ is a real pole of $S_{\mu}$.
 Let $E_{\{a\}}$ be the spectral orthogonal projection of $H_\Gamma$ corresponding to the one-point set $\{a\}$. Then $(E_{\{a\}}\delta_v,\delta_v)=\mu(\{a\})\neq 0$ and so $E_{\{a\}}\delta_v\neq 0$. One has $H_\Gamma E_{\{a\}}\delta_v=aE_{\{a\}}\delta_v$. Thus, $a$ is an eigenvalue of $H_\Gamma$.

Now assume that $a$ is a real pole of $S_{\mu}$. By Corollary \ref{CoSuppMu} and Remark \ref{RemConj}, $a\in\supp(\mu)$ and there is an open neighborhood $U\subset \mathbb R$ of $a$ such that $\mu(U\setminus\{a\})=0$. Thus, $a$ is an isolated atom of $\mu$.
\end{proof}
\noindent Note that, as example \ref{Ex5vertices} shows, it is possible even for a tree-like graph that $H_\Gamma$ has an eigenvalue $a$ which does not belong to $\supp(\mu_{\Gamma,v})$ for some vertex $v$.

The next statement shows that for the graphs under considerations supports of the measures $\mu_{\Gamma,v}$ for different vertices $v$ may differ only by isolated points representing eigenvalues. Let $A\bigtriangleup B$ stand for the symmetric difference of sets $A,B$.
\begin{Prop}\label{PropSuppMuEigenvectors} Let $\Gamma$ be a connected tree-like weighted graph of bounded degree with real symmetric uniformly bounded weights and $v,w$ be two vertices of $\Gamma$. Then $\supp(\mu_{\Gamma,v})\triangle\supp(\mu_{\Gamma,w})$ is a subset of the set of eigenvalues of $H_\Gamma$.
\end{Prop}
\begin{proof} Clearly, it is sufficient to prove the statement under the assumption that $v,w$ are two adjacent vertices of $\Gamma$. In addition, if $v$ and $w$ are connected by more than one edge we will replace these edges by one edge with the weight $\alpha_{(v,w)}$ equal to the sum of the weights between $v$ and $w$. The latter does not change either operator $H_\Gamma$ or measures $\mu_{\Gamma,v}$, $\mu_{\Gamma,w}$. Thus, we reduce the general case to the case with $v$ and $w$ connected by exactly one edge.

Let $\Gamma^v$ and $\Gamma^w$ be the connected components of the graph obtained by removing from $\Gamma$ the edge connecting $v$ and $w$. Recall that for a graph $\Omega$ and a vertex $u$ of $\Omega$ the notation $\Omega_u$ means the corresponding rooted graph with the root at $u$. One has:
 $$\Gamma_v=\Gamma_v^v\cup(\delta+\Gamma_w^w),\;\;\Gamma_w=(\delta+\Gamma_v^v)\cup\Gamma_w^w.$$
 Using Lemma \ref{LmGraphSum} we obtain:
\begin{equation}\label{EqFGammavvGammaww} F_{\Gamma,v}(t)=F_{\Gamma^v,v}(t)+\alpha_{(v,w)}^2t^2P_{\Gamma^w,w}(t),\;\;
 F_{\Gamma,w}(t)=F_{\Gamma^w,w}(t)+\alpha_{(v,w)}^2t^2P_{\Gamma^v,v}(t).\end{equation} Using Lemma \ref{LmPF} and \eqref{EqSmuPmu} setting $z=t^{-1}$ we arrive at the following equations:
 \begin{equation}\label{EqSmuGammavw}\frac{1}{S_{\mu_{\Gamma,v}}(z)}=
 \frac{1}{S_{\mu_{\Gamma^v,v}}(z)}-\alpha^2_{(v,w)}S_{\mu_{\Gamma^w,w}}(z),\;\;
 \frac{1}{S_{\mu_{\Gamma,w}}(z)}=
 \frac{1}{S_{\mu_{\Gamma^w,w}}(z)}-\alpha^2_{(v,w)}S_{\mu_{\Gamma^v,v}}(z),
 \end{equation} whenever all terms of the corresponding equation are defined. Let $I\subset \mathbb R$ be an open segment which does not intersect $\supp(\mu_{\Gamma,v})$. By definition of Cauchy-Stieltjes transform (see Subsection \ref{SubsecStiltjes}) the function $S_{\mu_{\Gamma,v}}(z)$ is real analytic on $I$ (\ie is analytic and has real values on $I$). In particular, the set $Z=\{z\in I:S_{\mu_{\Gamma,v}}(z)=0\}$ does not have any limit points in $I$.

 Further, for every $z_0\in I\setminus Z$ one has $$\lim\limits_{z\to z_0}\im\frac{1}{S_{\mu_{\Gamma,v}}(z)}=0.$$ Taking into account \eqref{EqSmuHalfplanes} from the first equation of \eqref{EqSmuGammavw} we obtain:
 $$\lim\limits_{z\to z_0,\im z>0}\im S_{\mu_{\Gamma^w,w}}(z)=0.$$ By Corollary \ref{CoSuppMu}, the latter implies that $I\setminus Z$ does not intersect $\supp(\mu_{\Gamma^w,w})$. It follows that $S_{\mu_{\Gamma^w,w}}(z)$ is real analytic on $I\setminus Z$. Therefore, the function $\frac{1}{S_{\mu_{\Gamma,v}}(z)}+\alpha_{(v,w)}^2S_{\mu_{\Gamma^w,w}}(z)$ is real analytic on $I\setminus Z$. By Remark \ref{RemConj} each of the points $z\in Z$ is either a removable singularity or a simple pole of $S_{\mu_{\Gamma^w,w}}$. We conclude that the set $$Z_1=\left\{z\in I:\frac{1}{S_{\mu_{\Gamma,v}}(z)}+\alpha_{(v,w)}^2S_{\mu_{\Gamma^w,w}}(z)=0,\;\;\text{or}\;\;S_{\mu_{\Gamma,v}}(z)=0,\;\;\text{or}\;\;S_{\mu_{\Gamma^w,w}}(z)=0\right\}$$ does not have any limit points in $I$.

 Now, using the first equation of \eqref{EqSmuGammavw} we obtain that $S_{\mu_{\Gamma^v,v}}(z)$ is nonzero and real analytic on $I\setminus Z_1$. Using the second equation of \eqref{EqSmuGammavw} we obtain that $\frac{1}{S_{\mu_{\Gamma,w}}(z)}$ is real analytic on $I\setminus Z_1$, and so $S_{\mu_{\Gamma,w}}(z)$ is real analytic on $I\setminus Z_1$ except possibly a set of poles without a limit point in $I$. By Remark \ref{RemConj} each of the points $z\in Z_1$ is either a point of analyticity or a simple pole of $S_{\mu_{\Gamma,w}}$. It follows that the set $Z_2$ of poles of $S_{\mu_{\Gamma,w}}$ in $I$ does not have limit points in $I$. From Corollary \ref{CoSuppMu} we obtain that $\supp(\mu_{\Gamma,w})\cap (I\setminus Z_2)=\varnothing$. Using Lemma \ref{LmAtomPoleEigenvalue} we derive that $Z_2$ is a subset of the set of eigenvalues of $H_\Gamma$.

 Thus, we obtain that $\supp(\mu_{\Gamma,w})\setminus \supp(\mu_{\Gamma,v})$ is a subset of the set of eigenvalues of $H_\Gamma$. Observing that the choice of the vertices $v$ and $w$ is symmetric, we finish the proof.
\end{proof}
The following folklore statement is a consequence of the application of Spectral Theorem to self-adjoint Laplace type operators.
\begin{Prop}\label{PropSpectrumUnion} Let $\Gamma$ be a weighted graph of bounded degree with real symmetric uniformly bounded weights. Then
$$\sigma(H_\Gamma)=\overline{\bigcup\limits_{v\in \Gamma}\supp(\mu_v)}.$$
\end{Prop}
\begin{proof} Let $\lambda\in \sigma(H_\Gamma)$. Then for any open interval $I$ containing $\lambda$ the corresponding spectral orthogonal projection $E_I$ of $H_\Gamma$ is nonzero, and so there exists a vertex $v$ such that $\mu_v(I)=(E_I\delta_v,\delta_v)\neq 0$. This means that $I\cap\supp(\mu_v)\neq\varnothing$. We obtain that $\lambda\in \overline{\bigcup\limits_{v\in \Gamma}\supp(\mu_v)}.$

On the other hand, if $\lambda\in\overline{\bigcup\limits_{v\in \Gamma}\supp(\mu_v)}$ then for any open interval $I$ containing $\lambda$ there exists a vertex $v$ of $\Gamma$ such that $\mu_v(I)\neq 0$, and so $E_I\neq 0$. This implies that $\lambda\in \sigma(H_\Gamma)$.
\end{proof}
Recall that $\Upsilon$ is the Schreier graph of the action of the Thompson group $F$ on the orbit of $1/2$ (see Figure \ref{FigSchreierThompson}) and $\Delta$ is the subgraph of $\Upsilon$ defined in Section \ref{SubsecWeightedSchreier} (see Figure \ref{FigSchreierThompsonSelfsim}).
\begin{Prop} \label{PropNoEigenvalues} Let $\alpha,\beta\in\mathbb R\setminus \{0\}$ be of the same sign. Then the operators $H_\Upsilon$ and $H_\Delta$ have no eigenvalues.
\end{Prop}
\begin{proof} Let us show that $H_\Upsilon$ has no eigenvalues. Dividing all weights of $H_\Upsilon$ by $\alpha$ we reduce the general case to the case $\alpha=1$. Assume that $H_\Upsilon$ has an eigenvalue $\lambda$. Note that $\lambda\in\mathbb R$ since $H_\Upsilon$ is self-adjoint. Let $f\in l^2(\Upsilon)$ be an eigenvector corresponding to $\lambda$. In particular, $f$ is not identically equal to zero. Let $u$ be a vertex of $\Upsilon$ such that $f(u)\neq 0$. The vertex $u$ has a copy of at least one of the graphs $\Phi$ and $\Psi$ attached to it (see Figure \ref{FigSchreierThompson}). In fact, it is not hard to see that both cases occur for vertices $u$ with $f(u)\neq 0$.

{\bf Case 1.} Assume that a copy of $\Psi$ is attached to $u$. We label the vertices of this copy of $\Psi$ by negative integers as in Subsection \ref{SubsecWeightedSchreier}. Then the values of $f$ satisfy the equation:
$$(1+\beta)f(-j)-\lambda f(-j-1)+(1+\beta)f(-j-2)=0,\;\;j\in\mathbb N.$$ The roots of the characteristic equation of this recursion are $$t_{1,2}=\frac{\lambda\pm\sqrt{\lambda^2-4(1+\beta)^2}}{2(1+\beta)}.$$ It follows that $$f(-j)=c_1t_1^j+c_2t_2^j,\;\;j\in\mathbb N, $$ where $c_1,c_2$ are constants. Note that $t_1t_2=1$. Moreover, if $|\lambda|\leqslant 2(1+\beta)$ then $|t_1|=|t_2|=1$. If $|\lambda|>2(1+\beta)$ then only one of the roots $$t_1:=\frac{\lambda-\sign(\lambda)\sqrt{\lambda^2-4(1+\beta)^2}}{2(1+\beta)}$$ has absolute value less than $1$. Since $f\in l^2(\Upsilon)$ we obtain that $|\lambda|>2(1+\beta)$ and $f(j)=c_1t_1^j$.

Further, the vertex $u=1$ has three neighbors: the vertex labeled by $2$ and two other vertices (possibly, coinciding) which we denote by $u_1$ and $u_2$ correspondingly. We assume that the edge connecting $u$ and $u_1$ has weight $1$ and the edge connecting $u$ and $u_2$ has weight $\beta$. One has:
$f(u_1)+\beta f(u_2)+(1+\beta)f(2)=\lambda f(1)$. Taking into account that $|\lambda|>2(1+\beta)$ and $|f(2)|<|f(1)|$ we obtain that $$\min\{|f(u_1)|,|f(u_2)|\}>|f(1)|. $$

{\bf Case 2.} Assume that a copy of $\Phi$ is attached to $u$. We label the vertices of this copy of $\Phi$ by positive integers as in Subsection \ref{SubsecWeightedSchreier}. Then the values of $f$ satisfy the equation:
$$f(j)+(2\beta-\lambda) f(j+1)+f(j+2)=0,\;\;j\in\mathbb N.$$ Solving the recurrent relation we obtain $$f(j)=c_1t_1^j+c_2t_2^j,\;\;t_{1,2}=\frac{\lambda-2\beta\pm\sqrt{(\lambda-2\beta)^2-4}}{2},
\;\;j\in\mathbb N,$$ where $c_1,c_2$ are constants. Observe that if $|\lambda-2\beta|\leqslant 2$ then $|t_1|=|t_2|=1$. If $|\lambda-2\beta|>2$ then only one of the roots $$t_1:=\frac{\lambda-2\beta-\sign(\lambda-2\beta)\sqrt{(\lambda-2\beta)^2-4}}{2}$$ has absolute value less than $1$. Since $f\in l^2(\Upsilon)$ we obtain that $|\lambda-2\beta|>2$ and $f(j)=c_1t_1^j$.

Further, let the neighbors of $u$ outside the copy of $\Phi$ be $u_1,u_2,u_3$ (some of them might coincide), where the edge between $u$ and $u_3$ has weight $1$. Then $$\beta f(u_1)+\beta f(u_2)+f(u_3)+f(2)=\lambda f(1).$$ Taking into account that $|\lambda|>2(1+\beta)$ (see case 1), $\beta>0$ and $|f(2)|<|f(1)|$ we obtain that $\min\{|f(u_1)|,|f(u_2)|,|f(u_3)|\}>|f(1)|$.

Thus, for every vertex $u$ such that $f(u)\neq 0$ there is an adjacent vertex $w$ such that $|f(w)|>|f(u)|$. This implies that $f\notin l^2(\Upsilon)$. This contradiction shows that $H_\Upsilon$ has no eigenvalues.

The case of operator $H_\Delta$ can be treated similarly. The only difference from the above arguments is that in case $2$ it is possible that $u=1$ have only two neighbors $u_1,u_2$ outside $\Phi$ (precisely when $u=v$, see Figure \ref{FigSchreierThompsonSelfsim}). Then $\beta f(u_1)+f(u_2)+f(2)=\lambda f(1)$ and we conclude that $\min\{|f(u_1)|,|f(u_2)|\}>|f(1)|$.
\end{proof}
For a graph $\Gamma$ and a vertex $v$ of $\Gamma$ let $B_n(v)$ stand for the ball of radius $n$ around $v$ in $\Gamma$. That is, $B_n(v)=(V_n(v),E_n(v))$, where $V_n(v)$ is the set of all vertices which can be reached by a path from $v$ of length at most $n$ in $\Gamma$, $E_n(v)$ is the set of all edges from $\Gamma$ with both endpoints in $V_n(v)$.
\begin{Def}\label{DefLocCont}For two uniformly bounded weighted graphs $\Gamma_1,\Gamma_2$ we will say that $\Gamma_1$ is locally contained in $\Gamma_2$ if for every vertex $v\in\Gamma_1$ and every $n\in\mathbb N$ there exists $w\in \Gamma_2$ such that $B_n(v)\subset \Gamma_1$ is isomorphic as a weighted graph to $B_n(w)\subset\Gamma_2$.\end{Def}

Let us recall a useful Lemma from \cite{DudkoGrigorchuk-Erratum-20}.
\begin{Lm}\label{LmEquivNorm} Let $A$ be any bounded nonzero linear operator on a Hilbert space and $R\geqslant 2\|A\|$. Then the following assertions are equivalent:
\begin{itemize} \item[$1)$] $\lambda\in \sigma(A)$,
\item[$2)$] $1\in \sigma(\mathrm{I}-
\tfrac{1}{R^2}(A-\lambda\mathrm{I})(A-\lambda\mathrm{I})^{*})\cup \sigma(\mathrm{I}-
\tfrac{1}{R^2}(A-\lambda\mathrm{I})^*(A-\lambda\mathrm{I})),$
\end{itemize}
where $\mathrm{I}$ is the identity operator.
\end{Lm}
\noindent The following statement is a straightforward generalization of Proposition 11 from \cite{DudkoGrigorchuk-Spectrum-17} (which also appears in \cite{DudkoGrigorchuk-Erratum-20} as Proposition 1).
\begin{Prop}\label{PropLocCont} Let $\Gamma_1,\Gamma_2$ be two graphs of bounded degree with uniformly bounded weights such that $\Gamma_1$ is locally contained in $\Gamma_2$. Then $\sigma(H_{\Gamma_1})\subset\sigma(H_{\Gamma_2})$.
\end{Prop}
\begin{proof}
Let $\lambda\in\sigma(H_{\Gamma_1})$. Take $R>2\max\{\|H_{\Gamma_1}\|,\|H_{\Gamma_2}\|\}$.
Using Lemma \ref{LmEquivNorm} we obtain that $$1\in \sigma(\mathrm{I}-
\tfrac{1}{R^2}(H_{\Gamma_1}-\lambda\mathrm{I})(H_{\Gamma_1}-\lambda\mathrm{I})^{*})\cup \sigma(\mathrm{I}-
\tfrac{1}{R^2}(H_{\Gamma_1}-\lambda\mathrm{I})^{*}(H_{\Gamma_1}-\lambda\mathrm{I})).$$ Assume, for instance, that $1\in \sigma(\mathrm{I}-
\tfrac{1}{R^2}(H_{\Gamma_1}-\lambda\mathrm{I})(H_{\Gamma_1}-\lambda\mathrm{I})^{*})$ (the second case can be treated similarly). Denote by $V_i$, $i=1,2$, the vertex set of $\Gamma_i$.
Using the operations on graphs introduced in \cite{DudkoGrigorchuk-Erratum-20} we define for $i=1,2$ the graph $\widetilde\Gamma_i$ with the vertex set $V_i$ such that
$$H_{\widetilde\Gamma_i}=\mathrm{I}-
\tfrac{1}{R^2}(H_{\Gamma_i}-\lambda\mathrm{I})(H_{\Gamma_i}-\lambda\mathrm{I})^{*},\;\;
i=1,2.$$
Namely, we set
$$\widetilde\Gamma_i=1-
\tfrac{1}{R^2}(\Gamma_i-\lambda)\circ(\Gamma_i-\lambda)^*,\;\;i=1,2.$$
We refer the reader to Lemma 4 in \cite{DudkoGrigorchuk-Erratum-20} for details.

The graph $\widetilde\Gamma_1$ is locally contained in the graph $\widetilde\Gamma_2$. The operators $H_{\widetilde\Gamma_i},i=1,2$ are positive of norm at most $1$. In fact, $\|H_{\widetilde\Gamma_1}\|=1$ since $1\in H_{\widetilde\Gamma_1}$ by our assumptions. Let $\epsilon>0$. Since
$$\sup\limits_{\xi:\|\xi\|=1} (H_{\widetilde\Gamma_1}\xi,\xi)=1$$ we can find $v\in V_1$,  $l\in\mathbb N$ and a vector $\eta\in l^2(V_1)$
 supported on $B_l(v)\subset \widetilde\Gamma_1$ such that $(H_{\widetilde\Gamma_1}\eta,\eta)>1-\epsilon$. Let
$w\in V_2$ be such that $B_{l+1}(w)\subset \widetilde\Gamma_2$ is isomorphic (as a rooted weighted graph) to
 $B_{l+1}(v)$. Let $\eta'\in l^2(B_l(w))\subset l^2(V_2)$ be a copy of $\eta$ via this isomorphism.
We get:
 $$(H_{\widetilde\Gamma_2}\eta',\eta')=(H_{\widetilde\Gamma_1}\eta,\eta)>1-\epsilon.$$ Since $\epsilon>0$ is arbitrary it follows that $\|H_{\widetilde\Gamma_2}\|\geqslant 1$. By construction, $\|H_{\widetilde\Gamma_2}\|\leqslant 1$ and therefore $\|H_{\widetilde\Gamma_2}\|=1$. By positivity of $H_{\widetilde\Gamma_2}$ we have $1\in\sigma(H_{\widetilde\Gamma_2})$. Using Lemma \ref{LmEquivNorm} we obtain that $\lambda\in\sigma(H_{\Gamma_2})$ which finishes the proof.
\end{proof}

%%%%%%%%%%%%%%%%%%%%%%%%%
\section{Proof of Theorem \ref{ThSameSign}.}
Recall that in Theorem \ref{ThSameSign} we assume that $\alpha,\beta\in\mathbb R\setminus\{0\}$ are of the same sign and $x\in (0,1)$. Let $\Gamma=\Gamma_x$ be the corresponding weighted Schreier graph of the action of $F$ and $\mu_x=\mu_{m,x}$ be the spectral measure of $\Gamma$ corresponding to the vector $\delta_x$. We need to show that $$\supp(\mu_x)=\sigma(H_{\Gamma})=[-2|\alpha+\beta|,2|\alpha+\beta|].$$

From the description of Schreier graphs of the action of $F$ on $[0,1]$ (see \cite{Savchuk-SchreierThompson-15}, Theorem 3.3) it follows that $\Gamma$ contains a copy of $\Delta$. Thus, for some vertex $y\in\Gamma$ one has $$\Gamma_y=\Delta_v\cup \widetilde\Gamma_y,$$ where $\Gamma_y$ is the graph obtained from $\Gamma_x$ by moving the root to $y$ and $\widetilde\Gamma_y$ is some subgraph of $\Gamma_y$. We have:
$$F_{\Gamma,y}(t)=F_{\Delta,v}(t)+F_{\widetilde\Gamma,y}(t).$$ From Lemma \ref{LmPF} and \eqref{EqSmuPmu} we obtain:
\begin{equation}\label{EqSMuGammaId}\frac{1}{S_{\mu_{\Gamma,y}}(z)}=
z+\frac{1}{S_{\mu_{\Delta,v}}(z)}+\frac{1}{S_{\mu_{\widetilde \Gamma,y}}(z)}.
\end{equation}

Further, let $I\subset (-2|\alpha+\beta|,2|\alpha+\beta|)$ be an open segment. Let us show that $I\cap \supp(\mu_{\Gamma,y})\neq\varnothing$. Assume the contrary. Then $S_{\mu_{\Gamma,y}}(z)$ is real analytic (\ie has real values and is analytic) on $I$. Clearly, $S_{\mu_{\Gamma,y}}(z)$ is not identically equal to zero. Therefore, there exists an open segment $J\subset I$ such that $S_{\mu_{\Gamma,y}}(z)\neq 0$ on $J$. It follows that
$$\lim\limits_{z\to z_0,\im z>0}\im \frac{1}{S_{\mu_{\Gamma,y}}(z)}=0$$ for all $z_0\in J$. On the other hand, from Lemma \ref{LmImag}, Remark \ref{RemMuDel} and \eqref{EqSmuHalfplanes} we have for every $z_0\in (-2|\alpha+\beta|,2|\alpha+\beta|)$:
$$\liminf\limits_{z\to z_0,\im z>0}\im\frac{1}{ S_{\mu_{\Delta,v}}(z)}<0,\;\;\liminf\limits_{z\to z_0,\im z>0}\im \frac{1}{S_{\mu_{\widetilde\Gamma,y}}(z)}\leqslant 0,$$ where the limit is allowed to be $-\infty$. We obtain a contradiction to \eqref{EqSMuGammaId} from which we conclude that $I\cap\supp(\mu_{\Gamma,y})\neq\varnothing$. Since $I\subset (-2|\alpha+\beta|,2|\alpha+\beta|)$ is an arbitrary open segment and $\supp(\mu_{\Gamma,y})$ is closed we obtain that $$[-2|\alpha+\beta|,2|\alpha+\beta|]\subset \supp(\mu_{\Gamma,y})\subset \sigma(H_{\Gamma}).$$

On the other hand, since for every vertex of $\Gamma$ the sum of the weights of the edges adjacent to it is $2(\alpha+\beta)$ we have $\|H_{\Gamma}\|\leqslant 2|\alpha+\beta|$ and so $\supp(\mu_{\Gamma,y})\subset \sigma(H_\Gamma)\subset [-2|\alpha+\beta|,2|\alpha+\beta|]$.
 This finishes the proof.

\section{Addendum: on Hulanicki type theorem for graphs.}\label{SectionHulanicki}
In \cite{DudkoGrigorchuk-Shape-18} the authors showed the following:
\begin{Th}[Weak Hulanicki Theorem for Graphs]\label{ThGraphCov}  Let $\widetilde\Gamma$ be a connected graph of bounded degree with uniformly bounded weight. Assume that $\widetilde \Gamma$ covers a weighted graph $\Gamma$ such that either
\begin{itemize}\item[$a)$] $\widetilde\Gamma$ is amenable and $\Gamma$ is finite or
\item[$b)$] $\widetilde\Gamma$ has subexponential growth.
 \end{itemize} Let $\widetilde H,H$ be the Laplace type operators associated to $\widetilde\Gamma,\Gamma$. Then $\sigma(H)\subset \sigma(\widetilde H)$.
\end{Th}
\noindent They raised a question wether the part $a)$ can be proven without the assumption of finiteness of $\Gamma$. Here we answer this question by presenting a pair of amenable graphs $\widetilde\Gamma$ and $\Gamma$ such that $\widetilde\Gamma$ covers $\Gamma$ but the spectrum of $\Gamma$ is not contained in the spectrum of $\widetilde\Gamma$. See figure 1 for graphic representations of these graphs. This example shows that Theorem \ref{ThGraphCov} will not be true if dropping the condition of finiteness of $\Gamma$ in part $a)$.

 Denote by $\Gamma_{\mathbb N}$ the graph with the vertex set identified with $\mathbb N$ and two vertices $i,k$ connected by an edge if and only if $|i-k|=1$. Let $T_4$ be the infinite 4-regular tree (each vertex has degree 4). Denote by $\Gamma$ the graph obtained by adding two loops at vertex $1$ of $\Gamma_{\mathbb N}$. Denote by $\widetilde\Gamma$ the graph obtained by attaching a copy of $\Gamma_{\mathbb N}$ at every vertex of $T_4$. Then the map $\phi:\widetilde\Gamma\to\Gamma$ sending each vertex of $T_4$ to the vertex $1$ of $\Gamma$ and each copy of $\Gamma_{\mathbb N}$ from $\widetilde\Gamma$ identically onto the copy of $\Gamma_{\mathbb N}$ in $\Gamma$ defines a covering of graphs. We make each of the graphs described above a weighted graph by putting weight 1 on each edge. Observe that the graphs $\widetilde\Gamma$ and $\Gamma$ are amenable.
 \begin{figure}[h]\centering\includegraphics[width=0.9\linewidth]
{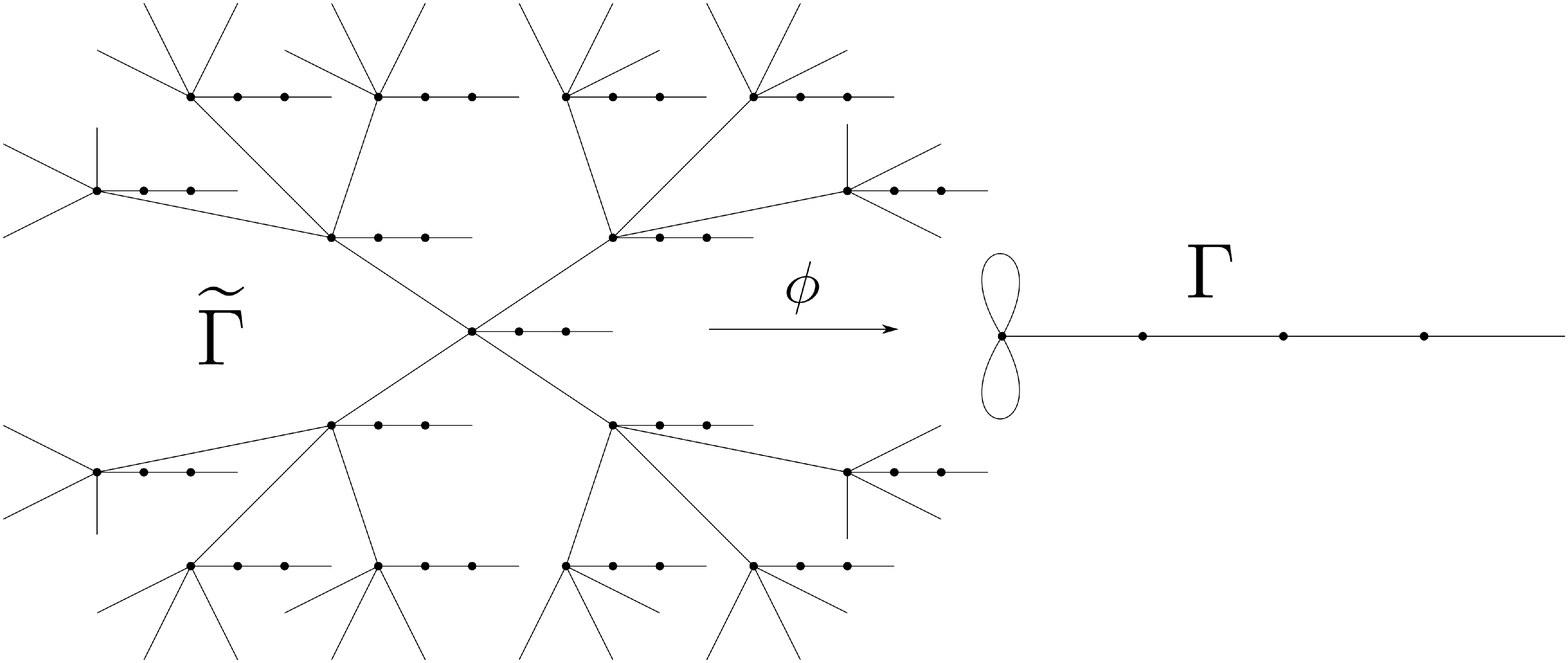}
\caption{The graphs $\Gamma$ and $\widetilde\Gamma$.}\label{FigGammaTildeGamma} \end{figure}
The main result of this section is
\begin{Th}\label{ThSpectraGamma} One has
$$\sigma(H_\Gamma)=[-2,2]\cup {\tfrac{17}{4}},\;\; \sigma(H_{\widetilde\Gamma})=\left[-\frac{13\sqrt 3}{6},\frac{13\sqrt 3}{6}\right].$$ In particular, $\sigma(H_\Gamma)$ is not a subset of $\sigma(H_{\widetilde\Gamma})$.
\end{Th}
\noindent Note that in the paper \cite{GrigorchukZuk-AsymptoticSpectrum-99}, Lemma 5, the authors show that for a uniform random walk on an infinite graph given that the Markov operator has spectral radius $1$ the eigenvalue $1$ cannot be isolated. In the proof of Lemma 5 it is important that the sum of the weights of edges adjacent to a vertex does not depend on the vertex. In our example the graph $\Gamma$ does not satisfy this condition. That is why it is possible for the Laplace type operator $H_\Gamma$ to have the largest by absolute value eigenvalue $\tfrac{17}{4}$ isolated.

\begin{Lm}\label{LmGammaN} One has $$F_{\Gamma_{\mathbb N},1}(t)=\tfrac{1}{2}(1-\sqrt{1-4t^2}),\;\;F_{\Gamma,1}(t)=4t+F_{\Gamma_{\mathbb N},1}(t).$$
\end{Lm}
\begin{proof} The weighted rooted graph $\Gamma_{\mathbb N,1}$ satisfies the equation
$$\Gamma_{\mathbb N,1}=\delta+\Gamma_{\mathbb N,1}\;\;(\text{see Subsection}\;\ref{SubsecGeneral}).$$ Using Lemmas \ref{LmPF} and \ref{LmGraphSum} we obtain that the function $f(t)=F_{\Gamma_{\mathbb N},1}(t)$ satisfies the equation
$$f(t)=\frac{t^2}{1-f(t)},\;\;f(0)=0,$$ solving which we arrive at $$f(t)=\tfrac{1}{2}(1-\sqrt{1-4t^2}).$$ Using Lemma \ref{LmGraphSum} we also get $F_{\Gamma,1}(t)=4t+F_{\Gamma_{\mathbb N},1}(t)$.
\end{proof}
\begin{Lm}\label{LmEigenvalues} Let $\Theta$ be a weighted graph, $v$ a vertex of $\Theta$ and $\Omega=\Theta_v\cup \Gamma_{\mathbb N,1}$. Assume that $H_\Omega$ has a real eigenvalue $\lambda$ and $f$ is a corresponding eigenvector. Then $|\lambda|>2$ and the values of $f$ on $\Gamma_{\mathbb N}\subset\Omega$ are of the form \begin{equation}\label{EqEigenvector}f(j)=c\left(\frac{\lambda-\sign(\lambda)\sqrt{\lambda^2-4}}{2}\right)^j,\;\;j\in\mathbb N,\end{equation} where $c$ is a constant and $\sign(\lambda)$ stands for the sign of $\lambda$.
\end{Lm}
\begin{proof} An eigenvector $f$ of $H_\Omega$ with an eigenvalue $\lambda$ has values on $\Gamma_{\mathbb N}$ satisfying:
\begin{equation}\label{EqRecursive}f(j+2)-\lambda f(j+1)+f(j)=0,\;\;j\in\mathbb N.\end{equation} The characteristic polynomial of this recursive equation $t^2-\lambda t+1=0$ has roots $$t_{1,2}=\frac{\lambda\pm\sqrt{\lambda^2-4}}{2}.$$ The general solution of \eqref{EqRecursive} is $$f(j)=c_1t_1^j+c_2t_2^j,$$ where $c_1,c_2$ are constants. Observe that for $\lambda\in[-2,2]$ we have $|t_1|=|t_2|=1$. Since $f$ is a vector from $l^2(\Omega)$, $f(j)\to 0$, we obtain in this casse that $c_1=c_2=0$. If $\lambda\notin [-2,2]$ then only one of the roots $t_1:=\frac{\lambda-\sign(\lambda)\sqrt{\lambda^2-4}}{2}$ has absolute value less than one. It follows that $c_2=0$ which finishes the proof of Lemma \ref{LmEigenvalues}.
\end{proof}
\begin{Lm}\label{LmSpectrumGamma} One has:
$$\sigma(H_{\Gamma_{\mathbb N}})=[-2,2],\;\;\sigma(H_\Gamma)=[-2,2]\cup {\tfrac{17}{4}}.$$
\end{Lm}
\noindent Note that the first part $\sigma(H_{\Gamma_{\mathbb N}})=[-2,2]$ is well known. For the readers convenience we present a proof here.
\begin{proof} Observe that $$F_{\Gamma_{\mathbb N},1}(t)=\tfrac{1}{2}(1-\sqrt{1-4t^2})$$ defines a function analytic on $\mathbb H_+=\{z\in\mathbb C:\im z>0\}$ and continuous on $\overline{\mathbb H_+}$ having real values for  $t\in [-1/2,1/2]$ and non-real values for $t\in\mathbb R\setminus [-1/2,1/2]$. The same property holds for $$P_{\Gamma_{\mathbb N},1}(t)=\frac{1}{1-F_{\Gamma_{\mathbb N},1}(t)},$$ since  $\tfrac{1}{2}(1-\sqrt{1-4t^2})\neq 1$ for $t\in\mathbb R$. Using \eqref{EqStiltGen} and Corollary \ref{CoSuppMu} we obtain that $\supp(\mu_1)=[-2,2]$, where $\mu_1$ is the spectral measure of $H_{\Gamma_{\mathbb N}}$ at vertex $1$.

Further, assume that $\Gamma_{\mathbb N}$ has an eigenvalue $\lambda$. By Lemma \ref{LmEigenvalues} $|\lambda|>2$ and the corresponding eigenvector $f$ is of the form \eqref{EqEigenvector}. Since $1$ has only one edge attached to it (leading to $2$) we get $\lambda f(1)=f(2)$ from which we get that the value of $c$ from \eqref{EqEigenvector} is zero and $f\equiv 0$. This contradiction shows that $H_{\Gamma_{\mathbb N}}$ does not have eigenvectors. Using Proposition \ref{PropSuppMuEigenvectors} we obtain that
$$\supp(\mu_j)=\supp(\mu_1)=[-2,2]\;\;\text{for every}\;\;j\in\mathbb N.$$ Finally, Proposition \ref{PropSpectrumUnion} implies that $\sigma(H_{\Gamma_{\mathbb N}})=[-2,2]$.

In case of the graph $\Gamma$ there are two differences. First, the equation
$F_{\Gamma,1}(t)=1$ has a unique real solution $t=\tfrac{4}{17}$. Indeed,
$$4t+\tfrac{1}{2}(1-\sqrt{1-4t^2})=1\;\Leftrightarrow\;8t-1=\sqrt{1-4t^2}\;\Rightarrow
\;68t^2-16t=0\;\Rightarrow\;t\in\{0,\tfrac{4}{17}\}.$$ From two solutions $t=0$ and $t=\tfrac{4}{17}$ only the second satisfies the initial equation. It follows that $P_{\Gamma,1}(t)=1/(F_{\Gamma,1}(t)-1)$ has a pole at $\tfrac{4}{17}$, is real analytic on $[-1/2,1/2]\setminus \{\tfrac{4}{17}\}$ and is non-real on $\mathbb R\setminus [-1/2,1/2]$. Using \eqref{EqStiltGen} and Corollary \ref{CoSuppMu} we obtain that $\supp(\mu_1)=[-2,2]\cup\{\tfrac{17}{4}\}$, where $\mu_1$ now is the spectral measure of $H_{\Gamma}$ at vertex $1$.

Second, for an eigenvector $f$ of $H_\Gamma$ corresponding to an eigenvalue $\lambda$ we have $\lambda f(1)=4f(1)+f(2)$. Together with \eqref{EqEigenvector} this leads to:
$$\lambda=4+\frac{\lambda-\sign(\lambda)\sqrt{\lambda^2-4}}{2}.$$ Setting $\lambda=t^{-1}$ we arrive at the equation considered above and conclude that $H_\Gamma$ has a unique eigenvalue $\lambda=\tfrac{17}{4}$. The corresponding eigenvector can be given by $f(j)=\frac{1}{4^j}$. Using Propositions \ref{PropSuppMuEigenvectors} and \ref{PropSpectrumUnion} we obtain that $\sigma(H_\Gamma)=[-2,2]\cup\{\tfrac{17}{4}\}$.
\end{proof}
\begin{figure}[h]\centering\includegraphics[width=0.6\linewidth]
{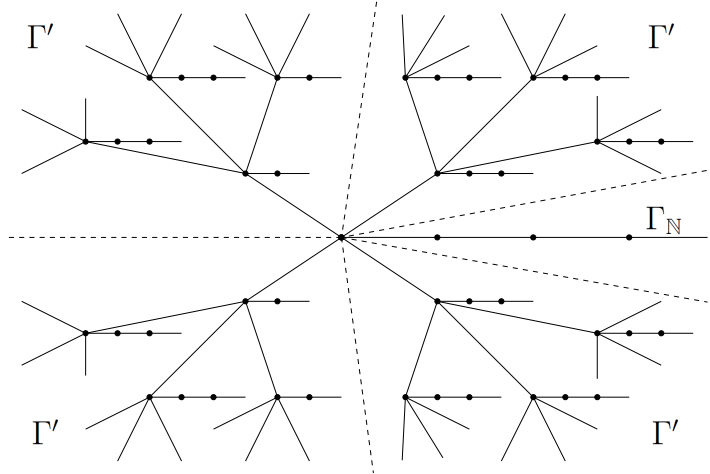}
\caption{Illustration to formula \eqref{EqTildeGammaPresentation}.}\label{FigTildeGammaSplitting} \end{figure}
Fix a vertex $v\in T_4\subset\widetilde\Gamma$. The graph $\widetilde\Gamma$ can be obtained as a union of four isomorphic copies of some graph $\Gamma'$ and a copy of $\Gamma_{\mathbb N}$ intersecting at $v$ (see Figure \ref{FigTildeGammaSplitting}):
\begin{equation}\label{EqTildeGammaPresentation}\widetilde\Gamma_v=\Gamma'_v\cup\Gamma'_v\cup\Gamma'_v\cup\Gamma'_v\cup\Gamma_{\mathbb N,1}.\end{equation}  Observe that $\Gamma'_v$ satisfies the equation:
\begin{equation}\label{EqGammavEquation}\Gamma'_v=\delta+(\Gamma'_v\cup\Gamma'_v\cup\Gamma'_v\cup\Gamma_{\mathbb N,1}).\end{equation}
\begin{Prop}\label{PropFtildeGamma} One has \begin{equation}\label{EqFGammaFormula}F_{\widetilde\Gamma,v}(t)=
\tfrac{1}{6}\left(5-\sqrt{1-4t^2}-2\sqrt{2-52t^2+2\sqrt{1-4t^2}}\right).
\end{equation}
\end{Prop}
\begin{proof} Using Lemmas \ref{LmPF} and \ref{LmGraphSum} we obtain from \eqref{EqGammavEquation}
$$F_{\Gamma',v}=\frac{t^2}{1-3F_{\Gamma',v}-F_{\Gamma_{\mathbb N},1}}.$$ Thus, $g=g(t):=F_{\Gamma',v}(t)$ satisfies the quadratic equation:
$$3g^2-(1-F_{\Gamma_{\mathbb N},1})g+t^2=0,$$ solving which taking into account that $g(0)=0$ we obtain:
\begin{align*}g=\tfrac{1}{6}\left(1-F_{\Gamma_{\mathbb N},1}-\sqrt{(1-F_{\Gamma_{\mathbb N},1})^2-12t^2}\right)=\\
\tfrac{1}{12}\left(1+\sqrt{1-4t^2}-\sqrt{2-52t^2+2\sqrt{1-4t^2}}\right).\end{align*}
From Lemma \ref{LmGraphSum} and \eqref{EqTildeGammaPresentation} we obtain $F_{\widetilde\Gamma,v}=4F_{\Gamma',v}+F_{\mathbb N,1}$ from which Proposition \ref{PropFtildeGamma} follows.
\end{proof}
\begin{Prop}\label{PropGammaNoEigenvalues} The operator $H_{\widetilde\Gamma}$ has no eigenvalues.
\end{Prop}
\begin{proof} Assume that $\lambda$ is an eigenvalue of $H_{\widetilde\Gamma}$. Let $f$ be the corresponding eigenvector. It is not hard to see that there exists a vertex $w\in T_4\subset \widetilde\Gamma$ such that $f(w)\neq 0$. For simplicity, identify the vertices of the copy of $\Gamma_{\mathbb N}$ containing $w$ with positive integers. In particular, $w$ is identified with 1. Let $w_1,w_2,w_3,w_4$ be the neighbors of $w$ in $T_4$. By Lemma \ref{LmEigenvalues}, $$f(2)=\gamma f(w),\;\;\text{where}\;\;\gamma=\frac{\lambda-\sign(\lambda)\sqrt{\lambda^2-4}}{2}.$$ On the other hand, $f(w_1)+f(w_2)+f(w_3)+f(w_4)+f(2)=\lambda f(w)$. It follows that
\begin{equation}\label{EqFW}f(w_1)+f(w_2)+f(w_3)+f(w_4)=(\lambda-\gamma)f(w).\end{equation}
The equation \eqref{EqFW} holds for any $w\in T_4$ such that $f(w)\neq 0$.

Now, let $w\in T_4$ be such that $f(w)=0$. Then by Lemma \ref{LmEigenvalues} the values of $f$ on the vertices of the corresponding copy of $\Gamma_{\mathbb N}$ are equal to zero. Therefore, $f(w_1)+f(w_2)+f(w_3)+f(w_4)=\lambda f(w)=0=(\lambda-\gamma)f(w)$. Thus, \eqref{EqFW} holds for all $w\in T_4$. This implies that $\lambda-\gamma$ is an eigenvalue of $H_{T_4}$. However, it is known that the Markov operator of the simple random walk on a $d$-regular tree does not have eigenvalues. This contradiction finishes the proof of Proposition \ref{PropGammaNoEigenvalues}.
\end{proof}
\begin{Prop} One has $\sigma(H_{\widetilde\Gamma})=[-\frac{2\sqrt 3}{13},\frac{2\sqrt 3}{13}]$.
\end{Prop}
\begin{proof} Consider the equation:
$$2-52t^2+2\sqrt{1-4t^2}=0\;\Rightarrow\;1-4t^2=(1-26t^2)^2\;\Rightarrow\;676t^4-48t^2=0\;\Rightarrow\; t\in\left\{0,\pm\frac{2\sqrt 3}{13}\right\}.$$ Using the latter we derive that $2-52t^2+2\sqrt{1-4t^2}$ has non-negative real values on $\left[-\frac{2\sqrt 3}{13},\frac{2\sqrt 3}{13}\right]$ and either negative or non-real values outside this interval. From formulae \eqref{EqFGammaFormula} we obtain that $F_{\widetilde\Gamma,v}(t)$ has real values on $\left[-\frac{2\sqrt 3}{13},\frac{2\sqrt 3}{13}\right]$ and non-real values outside this interval. The function $P_{\widetilde\Gamma,v}(t)=1/(1-F_{\widetilde\Gamma,v}(t))$ satisfies the same property, since $H_{\widetilde \Gamma}$ has no eigenvalues by Proposition \ref{PropGammaNoEigenvalues} and so $P_{\widetilde\Gamma,v}(t)$ has no poles by Lemma \ref{LmAtomPoleEigenvalue}. Using Corollary \ref{CoSuppMu} and \eqref{EqSmuPmu} we arrive at $$\supp(\mu_v)=\left[-\frac{13\sqrt 3}{6},\frac{13\sqrt 3}{6}\right].$$ Using Propositions \ref{PropSuppMuEigenvectors} and \ref{PropGammaNoEigenvalues} we obtain that $\supp(\mu_w)=\supp(\mu_v)$ for every vertex $w\in\widetilde\Gamma$. Finally, using Proposition \ref{PropSpectrumUnion} we derive that $\sigma(H_{\widetilde\Gamma})=\left[-\frac{13\sqrt 3}{6},\frac{13\sqrt 3}{6}\right]$.
\end{proof}

\bibliographystyle{siam}
\bibliography{Bibliography1}
\end{document}